\documentclass[11pt]{amsart}
\usepackage{a4wide}
\usepackage{amssymb}
\usepackage{mathrsfs,mathtools}
\usepackage{enumerate}
\usepackage{mathscinet}
\usepackage{titletoc}
\usepackage{esint}
\usepackage{mathscinet}
\usepackage[colorlinks=true, urlcolor=blue, linkcolor=blue, citecolor=green]{hyperref}
\usepackage[nameinlink]{cleveref}

\theoremstyle{plain}
\newtheorem{theorem}{Theorem}[section]
\newtheorem{lemma}[theorem]{Lemma}
\newtheorem{corollary}[theorem]{Corollary}

\theoremstyle{definition}
\newtheorem{definition}[theorem]{Definition}
\newtheorem{remark}[theorem]{Remark}

\numberwithin{equation}{section}

\DeclareMathOperator*{\osc}{osc}

\makeatletter
\@namedef{subjclassname@2020}{%
	\textup{2020} Mathematics Subject Classification}
\makeatother

\author{Sunghoon Kim}
\address{Department of Mathematics, The Catholic University of Korea, 43 Jibong-ro, Bucheon-si, Gyeonggi-do, 14662, Republic of Korea}
\email{math.s.kim@catholic.ac.kr}

\author{Ki-Ahm Lee}
\address{Department of Mathematical Sciences and Research Institute of Mathematics,
	Seoul National University, Seoul 08826, Republic of Korea}
\email{kiahm@snu.ac.kr}

\author{Se-Chan Lee}
\address{School of Mathematics, Korea Institute for Advanced Study, Seoul 02455, Republic of Korea}
\email{sechan@kias.re.kr}

\author{Minha Yoo}
\address{National Institute for Mathematical Sciences, Daejeon 34047, Republic of Korea}
\email{mhyoo@nims.re.kr}

\thanks{Sunghoon Kim was supported by the Research Fund 2022 of The Catholic University of Korea. Ki-Ahm Lee was supported by the National Research Foundation of Korea (NRF) grant funded by the Korea government (MSIP): NRF-2020R1A2C1A01006256. Se-Chan Lee was supported by the KIAS Individual Grant (No. MG099001) at the Korea Institute for Advanced Study. Minha You was supported by the National Institute for Mathematical Sciences (NIMS) granted by the Korean government (NIMS-B24910000)}

\begin{document}
\title[Homogenization of an obstacle problem]{Homogenization of an obstacle problem with highly oscillating coefficients and obstacles}


\begin{abstract}
	We develop the viscosity method for the homogenization of an obstacle problem with highly oscillating obstacles. The associated operator, in non-divergence form, is linear and elliptic with variable coefficients. We first construct a highly oscillating corrector, which captures the singular behavior of solutions near periodically distributed holes of critical size. We then prove the uniqueness of a critical value that encodes the coupled effects of oscillations in both the coefficients and the obstacles.
\end{abstract}

\maketitle


\section{Introduction}\label{sec-introduction}

In this paper, we are concerned with the homogenization of an obstacle problem with two different oscillating structures. To be precise, let $\Omega$ be a bounded $C^{1, 1}$-domain in $\mathbb{R}^n$ with $n \geq 3$. Throughout the paper, we suppose that $(a_{ij}(\cdot))_{1 \leq i, j \leq n}$ is a symmetric matrix-valued function defined in $\mathbb{R}^n$ satisfying the following conditions:
\begin{enumerate}[(i)]
	\item (Periodicity) $a_{ij}(y+k)=a_{ij}(y)$ for any $y \in \mathbb{R}^n$ and $k \in \mathbb{Z}^n$;
	
	\item (Uniformly ellipticity) there exist positive constants $\lambda \leq \Lambda$ such that
	\begin{equation*}
		\text{$\lambda|\xi|^2 \leq a_{ij}(y)\xi_i\xi_j \leq \Lambda|\xi|^2$ for any $y, \xi \in \mathbb{R}^n$;}
	\end{equation*}
	
	\item (Regularity) $a_{ij}(\cdot)$ is of {Dini mean oscillation}; see \Cref{sec-preliminaries} for the precise definition. 
\end{enumerate}
On the other hand, for sufficiently small $\varepsilon>0$, we define a set 
\begin{equation*}
	T_{a^{\varepsilon}}:=\bigcup_{k \in \mathbb{Z}^n}B_{a^{\varepsilon}}(\varepsilon k),
\end{equation*}
with the critical hole size $a^{\varepsilon}:=\varepsilon^{n/(n-2)}$. Moreover, for a smooth function $\varphi$ in $\Omega$ such that $\varphi \leq 0$ on $\partial \Omega$, we set highly oscillating obstacles $\varphi_{\varepsilon}$ by
\begin{equation*}
\varphi_{\varepsilon}(x):=
\left\{ \begin{array}{ll} 
\varphi(x) & \text{if $x \in T_{a^{\varepsilon}}$} \\
0 & \text{otherwise}.
\end{array} \right.
\end{equation*}

We now study an obstacle problem with highly oscillating coefficients and obstacles, that is, we are interested in the least viscosity supersolution $u_{\varepsilon}$ of the following problem:
\begin{equation}\label{problem-Le}\tag{$L_{\varepsilon}$}
\left\{ \begin{aligned} 
a_{ij}(x/\varepsilon) D_{ij}u_{\varepsilon}(x) &\leq 0 && \text{in $\Omega$} \\
u_{\varepsilon}&=0 && \text{on $\partial \Omega$} \\
u_{\varepsilon} &\geq \varphi_{\varepsilon} && \text{in $\Omega$}.
\end{aligned} \right.
\end{equation}
Our goal in the present paper is to describe the limit profile $u$ of $u_{\varepsilon}$ in a suitable convergence sense and to determine the effective equation satisfied by $u$. For this purpose, we define
	\begin{equation*}
		\underline{u}_{\varepsilon}(x):=
		\begin{cases}
			u_{\varepsilon}(x) & \text{if $x \in Q_{\varepsilon}(\varepsilon k) \setminus B_{b^{\varepsilon}}(\varepsilon k)$ for some $k \in \mathbb{Z}^n$}\\
			\min_{\partial B_{b^{\varepsilon}}(\varepsilon k)} u_{\varepsilon} & \text{if $x \in B_{b^{\varepsilon}}(\varepsilon k)$ for some $k \in \mathbb{Z}^n$},
	\end{cases}
	\end{equation*}
where $b^{\varepsilon}:=(\varepsilon a^{\varepsilon})^{1/2}$ and $Q_r(x_0):=x_0+[-r/2, r/2]^n$. Then our homogenization result reads as follows. 
\begin{theorem}\label{thm-main}
	Let $u_{\varepsilon}$ be the least viscosity supersolution of \eqref{problem-Le}.
	\begin{enumerate}[(i)]
		\item (Convergence) There exists a function $u \in C(\overline{\Omega})$ such that $\underline{u}_{\varepsilon} \to u$ uniformly in $\Omega$ as $\varepsilon \to 0$. In particular, $u_{\varepsilon} \to u$ in $L^p(\Omega)$ as $\varepsilon \to 0$ for any $p \geq 1$.
		
		\item (Homogenized equation) There exist a nonnegative value $\beta_0$ and a uniformly elliptic constant  matrix $\overline{a}_{ij}$ such that $u$ is a viscosity solution of
		\begin{equation}\tag{$\overline{L}$}\label{eq-homo}
		\left\{ \begin{aligned}
		\overline{a}_{ij} D_{ij}u+\beta_0 (\varphi-u)_+ &=0 && \text{in $\Omega$} \\
		u&=0 && \text{on $\partial \Omega$}.
		\end{aligned} \right.
		\end{equation}
	\end{enumerate} 
\end{theorem}
Let us make a couple of remarks on \Cref{thm-main}. First of all, the limit solution $u$ satisfies no longer an obstacle problem but an elliptic equation of a constant coefficient \textit{without an obstacle}. Roughly speaking, the effect of a rapidly oscillating obstacle with critical hole size is encoded in the additional term $\beta_0(\varphi-u)_+$, which is called a `strange term coming from nowhere' in \cite{CM82a, CM82b}. The precise description of $\overline{a}_{ij}$ and $\beta_0$ will be presented in \Cref{sec-correctors}. On the other hand, if the hole size is not critical, then the limit behaviors of $u_{\varepsilon}$ becomes relatively simple; see \Cref{sec-noncritical} for details.

Moreover, we observe that $u_{\varepsilon}$ has an oscillation of order $1$ in each cube $Q_{\varepsilon}$ of size $\varepsilon$, due to the obstacle condition given by $u_{\varepsilon} \geq \varphi$ in $T_{a^{\varepsilon}}$. In other words, the behavior of $u_{\varepsilon}$ is strongly restricted by the existence of an obstacle function $\varphi$, and so it is impossible to achieve the uniform convergence of $u_{\varepsilon}$ in the whole domain $\Omega$. We instead illustrate the uniform convergence of modified functions $\underline{u}_{\varepsilon}$, which capture the ``away from holes" behavior of $u_{\varepsilon}$ by introducing the intermediate scale $b^{\varepsilon}$.

The proof of \Cref{thm-main} consists of several steps. In fact, we begin with controlling the oscillation of $u_{\varepsilon}$ ``away from holes" in \Cref{lem-discrete} and \Cref{lem-cubeosc}. Then we demonstrate the uniform convergence of $\underline{u}_{\varepsilon}$ to some limit function $u$ up to subsequence in \Cref{thm-convergence}. We emphasize that the uniqueness of such limit $u$ is not guaranteed yet in this step, and it only corresponds to the partial result of \Cref{thm-main} (i). The uniqueness of $u$, the convergence in the full sequence, and the effective equation of $u$ will be established in \Cref{sec-homogenization}, after constructing appropriate correctors in \Cref{sec-correctors}. We also clarified several delicate issues that appeared in the previous literature such as \cite{CL08, LL23}, by providing the suitable choices of domains and points associated with the comparison principle. 

We next move our attention to the issue of correctors, which plays a crucial role in explaining the difference between $u_{\varepsilon}$ and $u$. The key step in most homogenization problems is to construct corrector functions having appropriate properties depending on the situation. For the setting of obstacle problems with two different oscillatory characters, we would like to construct a corrector $w_P^{\varepsilon}$ (for a quadratic polynomial $P$) satisfying
	\begin{enumerate}[(P1)]
		\item $a_{ij}D_{ij}P+a_{ij}D_{ij}w_P^{\varepsilon}=\overline{a}_{ij}D_{ij}P+\beta_0$;
        \item $w_P^{\varepsilon}\approx1$ on $\partial T_{a^{\varepsilon}}$;
		\item $w_P^{\varepsilon} \to 0$ ``away from holes" as $\varepsilon \to 0$.
	\end{enumerate}
	By exploiting the linear nature of \eqref{problem-Le}, we expect that the corrector $w_P^{\varepsilon}$ consists of essentially two parts: $w_{1, D^2P}$ for the coefficient oscillation (which describes $\overline{a}_{ij}$) and $w_2^{\varepsilon}$ for the obstacle oscillation (which describes $\beta_0$). Roughly speaking, $w_{1, D^2P}$ is a $1$-periodic function satisfying
	\begin{equation*}
		a_{ij}(y)D_{ij}P+a_{ij}(y)D_{ij}w_{1, D^2P}(y)=\overline{a}_{ij}D_{ij}P 
	\end{equation*}
	and $w_2^{\varepsilon}$ has an oscillation of order $1$ satisfying 
	\begin{equation*}
\left\{ \begin{aligned} 
a_{ij}(x/\varepsilon)D_{ij}w_{2}^{\varepsilon}(x)&=\beta_0 && \text{in $\mathbb{R}^n \setminus T_{a^{\varepsilon}}$} \\
w_{2}^{\varepsilon}&=1  && \text{on $\partial T_{a^{\varepsilon}}$}.
\end{aligned} \right.
\end{equation*}
See \Cref{sec-correctors} and \Cref{sec-homogenization} for the precise definition of $w_{1, D^2P}$, $w_2^{\varepsilon}$, and for the full structure of $w_P^{\varepsilon}$, respectively. Since the corrector $w_{1, D^2P}$ for the coefficient part can be taken from the one developed in \cite{Eva89, Eva92}, we would like to concentrate on the construction of $w_2^{\varepsilon}$ (or $w^{\varepsilon}$ for simplicity) as follows.

\begin{enumerate}[(i)]
    \item  We first show the existence of a value $\beta^{\varepsilon}>0$ such that the problem
\begin{equation*}
			\left\{ \begin{aligned} 
				a_{ij}(y)D_{ij}W^{\varepsilon}(y)&=\beta^{\varepsilon} && \text{in $\mathbb{R}^n \setminus \bigcup_{k \in \mathbb{Z}^n} B_{\overline{a}^{\varepsilon}}(k)$} \\
				W^{\varepsilon}&=\varepsilon^{-2}  && \text{on $\bigcup_{k \in \mathbb{Z}^n} \partial B_{\overline{a}^{\varepsilon}}(k)$}\\
				\inf W^{\varepsilon}&=0  && 
			\end{aligned} \right.
		\end{equation*}
		admits a unique $1$-periodic solution $W^{\varepsilon}$, where $\overline{a}^{\varepsilon}:=a^{\varepsilon}/\varepsilon$. We note that $W^{\varepsilon}$ is a scaled corrector in the fast variable $y=x/\varepsilon$. The key observation is that the desired shape of $W^{\varepsilon}$ looks like a singular solution with an isolated singularity. Therefore, we utilize Green's functions $G$ for the operator $L=a_{ij}D_{ij}$ to construct periodic viscosity super/subsolutions. See \Cref{rmk-green} for the definition of Green's functions.

  \item  We next prove that the unique limit $\beta_0$ of $\beta^{\varepsilon}$ exists, which is called the critical value. The proof of this step also relies on the investigation of the asymptotic behavior of $W^{\varepsilon}$ near each singularity point $k \in \mathbb{Z}^n$. In fact, this step becomes more involved compared to the one in the preceding literature for the constant coefficient, because an additional effort is required to analyze the effect of the coefficient oscillation. To be precise, (a) we first extend the function $W^{\varepsilon}-\gamma_0 G$, where $\gamma_0:=\mathrm{cap}_{a_{ij}(0)}(B_1)$ and $G$ is the Green's function for $a_{ij}D_{ij}$ in $B_{1/2}$, by employing the freezing coefficient argument in a different scale; (b) we then verify the extended function nicely converges to a bounded function without singularity in the scale of the fast variable, by establishing a uniform estimate in terms of the optimal $L^1$-data for properly decomposed solutions.

  \item We finally define the corrector $w_2^{\varepsilon}$ by scaling back. The special choice of $\beta_0$ allows us to transport appropriate properties of the periodic corrector $W^{\varepsilon}$ into the corrector $w_2^{\varepsilon}$ in the slow variable $x$. In particular, $w_2^{\varepsilon}$ converges to $0$ ``away from holes" when $\varepsilon$ goes to $0$; see \Cref{lem-convergence}.
\end{enumerate}

Let us now summarize several historical backgrounds related to our homogenization result. We refer to two comprehensive books \cite{BLP78, JKO94} and references therein for the overview. In particular, if an oscillation takes place either only for the coefficient or the obstacle, then the correctors and the corresponding homogenization results are relatively well studied. Let us begin with the homogenization of \eqref{problem-Le} without an obstacle. More precisely, if $u_{\varepsilon}$ solves the following $\varepsilon$-problems without obstacles
\begin{equation*}
	\left\{ \begin{aligned} 
		a_{ij}(x/\varepsilon) D_{ij}u_{\varepsilon}(x) &=f && \text{in $\Omega$} \\
		u_{\varepsilon}&=g && \text{on $\partial \Omega$},
		\end{aligned}\right.
\end{equation*}
then $u_{\varepsilon}$ converge to a limit profile $u$ as $\varepsilon \to 0$, which solves a problem
\begin{equation*}
	\left\{ \begin{aligned} 
		\overline{a}_{ij} D_{ij}u_{\varepsilon} &=f && \text{in $\Omega$} \\
		u&=g && \text{on $\partial \Omega$}
	\end{aligned}\right.
\end{equation*}
for some constant matrix $\overline{a}_{ij}$; see \cite{Eva89} for details. We also refer to \cite{Eva92} for the fully nonlinear version of such result, \cite{KL16} for the higher order convergence rates, \cite{AL89, KL22} for the uniform estimates, and \cite{AKM17, AL17, CS10, CSW05} for the convergences and convergence rates in the stochastic framework.  

On the other hand, the homogenization of an obstacle problem whose coefficient does not oscillate in the microscopic scale is known in the literature, such as \cite{CM82a, CM82b} via the energy method and \cite{CL08} via the viscosity method. To be precise, if $u_{\varepsilon}$ solves the problem \eqref{problem-Le} with a Laplacian operator $a_{ij}(\cdot)\equiv \delta_{ij}$, then the limit $u=\lim_{\varepsilon \to 0}u_{\varepsilon}$ solves a non-obstacle problem \eqref{eq-homo} with the same coefficient $\overline{a}_{ij}=\delta_{ij}$. We would like to point out that, in this Laplacian case, the critical value $\beta_0$ is explicitly given by the capacity of a unit ball; see \cite{CL08, CM82a, CM82b} for details. For the critical value $\beta_0$ in the variable coefficient case, we are going to provide a remark at \Cref{sec-homogenization} together with a concrete example. Moreover, similar consequences can be found in \cite{CM09, LL23, Tan12} for the stationary ergodic framework, \cite{KL11} for the porous medium equations, and \cite{LY12} for the soft inclusion problem.

The paper is organized as follows. In \Cref{sec-preliminaries}, we collect important preliminary results regarding Green's functions. \Cref{sec-limit} consists of the fundamental properties of the $\varepsilon$-solutions $u_{\varepsilon}$ which are strongly related to the convergence of $u_{\varepsilon}$. In \Cref{sec-correctors}, we first present known results on the correctors for coefficients and then construct correctors for obstacles together with proving the uniqueness of the critical value $\beta_0$. \Cref{sec-homogenization} is devoted to the proof of \Cref{thm-main}. Finally, we shortly describe the behavior of solutions when the hole size is non-critical in \Cref{sec-noncritical}.

\section{Preliminaries}\label{sec-preliminaries}
Since the construction of correctors strongly relies on the viscosity method, the starting point is to capture the asymptotic behavior of oscillating obstacles $\varphi_{\varepsilon}$ near holes. For this purpose, we introduce a Green's function for a uniformly elliptic and linear operator $L$ in non-divergence form. We begin with the definition of Green's functions in the non-divergent framework.
\begin{remark}[\cite{Bau84}]\label{rmk-green}
    By the $L^p$-theory (see \cite[Chapter 9]{GT83} for instance), for each $f \in L^p(\Omega)$ with $p>n$, there exists a unique function $u \in W^{2, p}_{\mathrm{loc}}(\Omega) \cap C(\overline{\Omega})$ satisfying
    \begin{equation}\label{eq-dirichletproblem}
		\left\{ \begin{aligned}
		Lu={a}_{ij} D_{ij}u &=-f && \text{in $\Omega$} \\
		u&=0 && \text{on $\partial \Omega$},
		\end{aligned} \right.
    \end{equation}
       together with the estimate
       \begin{equation*}
           \|u\|_{L^{\infty}(\Omega)} \leq C\|f\|_{L^{p}(\Omega)}.
       \end{equation*}
       Thus, for each fixed $y \in \Omega$, the mapping $f \mapsto u(y)$ is a continuous positive linear functional on $L^p(\Omega)$. The Riesz representation theorem then implies the existence of a nonnegative function $G(y, \cdot) \in L^{p/(p-1)}(\Omega)$ such that 
       \begin{equation*}
        u(y)= \int_{\Omega}G(y, x)f(x)\,dx.   
       \end{equation*}
        The function $G(\cdot, \cdot)$ is called the \textit{Green's function} for $L$ in $\Omega$.
\end{remark}

We then recall the definition of Dini mean oscillation on the coefficient matrix $A=(a_{ij})$, which guarantees nice properties of Green's functions for linear equations in non-divergence form. 
\begin{definition}
    We say that a function $g : \Omega \to \mathbb{R}$ is of \textit{Dini mean oscillation} in $\Omega$ if the mean oscillation function $\omega_g :  \mathbb{R}_+ \to \mathbb{R}$ defined by
    \begin{equation*}
        \omega_g(r):=\sup_{x \in \overline{\Omega}} \, \fint_{\Omega \cap B_r(x)} |g(y)-\overline{g}_{\Omega \cap B_r(x)}|\, dy, \quad \text{where}  \quad \overline{g}_{\Omega \cap B_r(x)}:=\fint_{\Omega \cap B_r(x)}g(y)\,dy
    \end{equation*}
    satisfies the Dini condition, i.e.,
    \begin{equation*}
        \int_0^1 \frac{\omega_g(t)}{t} \,dt <\infty.
    \end{equation*}
\end{definition}

\begin{remark}\label{rmk-dini}
    We remark that the following implications hold:
    \begin{equation*}
    \begin{aligned}
        &\text{$g$ is $\alpha$-H\"older continuous for some $\alpha \in (0,1)$} \implies \text{$g$ is Dini continuous} \\
        &\quad \implies \text{$g$ is of Dini mean oscillation} \implies  \text{$g$ is uniformly continuous}.
    \end{aligned}
    \end{equation*}
    See \cite[Section 1]{HK20} for instance. Moreover, if  $a_{ij}$ is of Dini mean oscillation, then strong solutions $u$ of $a_{ij}D_{ij}u=0$  belong to $C^2$-space; see \cite[Theorem 1.6]{DK17} and \cite[Theorem 1.5]{DEK18}.
\end{remark}

We now turn our attention to Green's functions; consider an approximated Green's function $G^{\sigma}(y, x_0)$ for $\sigma \in (0,1)$ and $y, x_0 \in \mathbb{R}^n$, as in \cite{GW82, HK20}:
\begin{equation*}
	\left\{
	\begin{aligned}
		a_{ij}(y)D_{ij}G^{\sigma}(y, x_0)&=-f_{\sigma}(y, x_0):=-\frac{1}{|B_{\sigma}(x_0)|}\chi_{B_{\sigma}(x_0)}(y) \quad && \text{for $y \in B_1(x_0)$}\\
		G^{\sigma}(y, x_0)&=0 \quad && \text{for $y \in \partial B_1(x_0)$}.
	\end{aligned}\right.
\end{equation*}
Since the forcing term $f_{\sigma}(\cdot, x_0) \in L^{\infty}(B_1)$, the $L^p$-theory yields that $G^{\sigma}(\cdot, x_0) \in W^{2, p}_{\text{loc}}(B_1)$ for any $p \in (1, \infty)$. Here we understand $G^{\sigma}$ as a strong solution of the Dirichlet problem.

\begin{lemma}[{\cite[Section 2.1]{HK20}}]\label{lem-green1}
	Suppose that $\Omega$ is a bounded $C^{1, 1}$-domain in $\mathbb{R}^n$ with $n \geq 3$. Assume that the coefficient $A=(a_{ij})$ is uniformly elliptic and of Dini mean oscillation in $\Omega$. Then the following hold:
	\begin{enumerate}[(i)]
		\item (Convergence I) $G^{\sigma}(\cdot, x_0) \rightharpoonup G(\cdot, x_0)$ weakly in $W^{2, 2}(\Omega \setminus B_{r}(x_0))$ up to subsequence, for any $r>0$.
		
		\item (Convergence II) $G^{\sigma}(\cdot, x_0) \to G(\cdot, x_0)$ uniformly on $\Omega \setminus B_{r}(x_0)$ up to  subsequence, for any $r>0$.
		
		\item (Upper bound) Let $x, y \in \Omega$ with $x \neq y$. Then
		\begin{equation*}
			G^{\sigma}(x, y) \leq C|x-y|^{2-n} \quad \text{for any $\sigma \in (0, |x-y|/3)$.}
		\end{equation*}
	In particular, $G(x, y) \leq C|x-y|^{2-n}$ for $C=C(n, \lambda, \Lambda, \Omega, \omega_{A})>0$, where $\omega_A$ denotes the mean oscillation function of $A=(a_{ij})$.
	\end{enumerate}
\end{lemma}
It immediately follows from \Cref{lem-green1} that
\begin{equation*}
	\left\{
	\begin{aligned}
		a_{ij}(y)D_{ij}G(y, x_0)&=0 \quad && \text{for $y \in B_1(x_0) \setminus \{x_0\}$}\\
		G(y, x_0)&=0 \quad && \text{for $y \in \partial B_1(x_0)$}.
	\end{aligned}\right.
\end{equation*}
Moreover, by employing the perturbation theory with respect to the constant coefficient operator, we can investigate the asymptotic behavior of $G(x, x_0)$ near $x_0$. Indeed, we are going to utilize this `almost' homogeneous property to replace the homogeneity.
\begin{lemma}[{\cite[Theorem 1.11]{DKL23}}]\label{lem-green2}
	Under the same hypothesis of \Cref{lem-green1}, we have
	\begin{equation*}
		\begin{aligned}
		&\lim_{x \to x_0} |x-x_0|^{n-2}|G(x, x_0)-G_{x_0}(x, x_0)|=0,\\
		&\lim_{x \to x_0} |x-x_0|^{n-1}|D_xG(x, x_0)-D_xG_{x_0}(x, x_0)|=0,\\
		&\lim_{x \to x_0} |x-x_0|^{n}|D^2_xG(x, x_0)-D^2_xG_{x_0}(x, x_0)|=0,
		\end{aligned}
	\end{equation*}
for any $x_0 \in \Omega$.
\end{lemma}
We note that the Green's function $G_{x_0}$ for a constant coefficient operator $L_{x_0}=a_{ij}(x_0)D_{ij}$ satisfies
\begin{equation}\label{eq-green-const}
	c_1 |x-x_0|^{2-n} \leq G_{x_0}(x, x_0)\leq c_2 |x-x_0|^{2-n}.
\end{equation}
Therefore, by combining \Cref{lem-green2} and \eqref{eq-green-const}, we obtain
\begin{equation}\label{eq-almost}
	\frac{c_1}{2}|x-x_0|^{2-n} \leq G(x, x_0) \leq 2c_2|x-x_0|^{2-n}, 
\end{equation}
provided that $|x-x_0|$ is small enough. We can understand this property as the \textit{almost homogeneity} of $G$.

We finish this section by developing uniform estimates for solutions of \eqref{eq-dirichletproblem} for the later use in \Cref{sec-correctors}. Let $u$ be the unique solution of \eqref{eq-dirichletproblem}, where $f \in L^p(\Omega)$ for $1<p<\infty$. Then by the standard $L^p$-theory, one can obtain a uniform estimate
\begin{equation*}
    \|u\|_{W^{2, p}(\Omega)} \leq C \|f\|_{L^p(\Omega)}.
\end{equation*}
Furthermore, we can prove the uniform estimate in terms of $\|f\|_{L^1(\Omega)}$ as follows.
\begin{theorem}\label{thm-L1data}
	Let $u$ be the solution of \eqref{eq-dirichletproblem}, where $f \in L^p(\Omega)$ for $p>n$. Then we have the following estimate
	\begin{equation*}
		\|u\|_{W^{1, q}(\Omega)} \leq C\|f\|_{L^1(\Omega)}
	\end{equation*}
 for any $1\leq q<n/(n-1)$.
\end{theorem}

\begin{proof}
	By applying \Cref{rmk-green}, we have a representation formula for $u$:
 \begin{equation*}
		u(y)=\int_{\Omega}G(y, x)f(x) \,dx,
	\end{equation*}
 where $G$ is the Green's function for $L$ in $\Omega$. Then an application of \cite[Theorem 1.9]{HK20} yields that
	\begin{equation*}
		|\nabla u(y)| \leq \int_{\Omega} |\nabla_y G(y, x)| |f(x)|\,dx \leq C\int_{\Omega} |x-y|^{1-n} |f(x)|\,dx,
	\end{equation*}
	and so by utilizing the Riesz potential estimate (see \cite[Lemma 7.12]{GT83} for instance),
	\begin{equation*}
		\|\nabla u\|_{L^q(\Omega)} \leq  C\|f\|_{L^1(\Omega)} \quad \text{for any $1 \leq q<\frac{n}{n-1}$}.
	\end{equation*}
\end{proof}

\section{Almost uniform convergence}\label{sec-limit}
The aim of this section is to prove the partial result of \Cref{thm-main} (i), i.e., the existence of a limit function $u$ up to subsequence of $\varepsilon$. We note that the consequences in this section also hold for non-critical hole sizes, since the proofs do not exploit the specific choice of $a^{\varepsilon}$.

We begin with the discrete gradient estimate, which gives a nice control for $|u_{\varepsilon}(x_1)-u_{\varepsilon}(x_2)|$ whenever $x_1-x_2 \in \varepsilon\mathbb{Z}^n$.
\begin{lemma}[Discrete gradient estimate]\label{lem-discrete}
		For given direction $e=e_l$ with $l=1, \cdots, n$, set
		\begin{align*}
			\Delta_e u_{\varepsilon}(x)=\frac{u_{\varepsilon}(x+\varepsilon e)-u_{\varepsilon}(x)}{\varepsilon}.
		\end{align*}
		Then there exists a universal constant $C>0$ which is independent of $\varepsilon>0$ such that
		\begin{align*}
			|\Delta_e u_{\varepsilon}(x)| \leq C
		\end{align*}
		for every $x \in \widetilde{\Omega}_{\varepsilon}:=\Omega_{\varepsilon} \cap (\varepsilon e+\Omega_{\varepsilon})$.
	\end{lemma}
	
	\begin{proof}
		Since $a_{ij}(\cdot)$ is $1$-periodic, the proof is essentially the same as the one of \cite[Lemma 3.1]{CL08} or \cite[Lemma 2.5]{KL11}.
	\end{proof}
	
We next show that the oscillation of $u_{\varepsilon}$ in each cube ``away from holes'' is sufficiently small.
\begin{lemma}[$\varepsilon$-flatness]\label{lem-cubeosc}
	Set $b^{\varepsilon}=(\varepsilon a^{\varepsilon})^{1/2}$. Then there exists a universal constant $C>0$ such that
	\begin{equation*}
		\osc_{Q_{\varepsilon}(\varepsilon k) \setminus B_{b^{\varepsilon}}(\varepsilon k)} u_{\varepsilon} \leq C({a}^{\varepsilon}/\varepsilon)^{\frac{n-2}{2}}+C\varepsilon
	\end{equation*}
	for each $k \in \mathbb{Z}^n$ satisfying $Q_{3\varepsilon}(\varepsilon k) \subset \Omega$.
\end{lemma}

\begin{proof}
	Without loss of generality, we may assume $k=0$.
	Let 
	\begin{equation*}
		m_{\varepsilon}:=\min_{Q_{3\varepsilon} \cap (T_{\varepsilon/2} \setminus T_{\varepsilon/100})}u_{\varepsilon} \quad \text{for $T_{r}:=\bigcup_{k \in \mathbb{Z}^n} B_r(\varepsilon k)$,}
	\end{equation*}
	and choose $x_0^{\varepsilon} \in Q_{3\varepsilon} \cap (T_{\varepsilon/2}\setminus T_{\varepsilon/100})$ such that $u_{\varepsilon}(x_0^{\varepsilon})=m_{\varepsilon}$. Then by applying \Cref{lem-discrete} if necessary, there exists a point $x^{\varepsilon} \in Q_{\varepsilon} \cap (B_{\varepsilon/2}\setminus B_{\varepsilon/100})$ such that
	\begin{equation*}
		m_{\varepsilon} \leq u_{\varepsilon}(x^{\varepsilon}) \leq m_{\varepsilon}+C\varepsilon.
	\end{equation*}
	We note that $\varepsilon/100 \leq r_{\varepsilon}:=|x^{\varepsilon}| \leq \varepsilon/2$.
	
	We now claim that 
	\begin{equation}\label{eq-claim-lower}
		u_{\varepsilon} \geq m_{\varepsilon}\quad \text{in $Q_{\varepsilon}$}.
	\end{equation}
	\begin{enumerate}[(i)]
		\item For $x \in  Q_{\varepsilon} \cap (B_{\varepsilon/2}\setminus B_{\varepsilon/100})$, it holds by  the definition of $m_{\varepsilon}$.
				
		\item For $x \in Q_{\varepsilon} \setminus B_{\varepsilon/2}$ or $x \in B_{\varepsilon/100}$, it follows from the minimum principle.
	\end{enumerate}
	
	Thus, the desired lower bound \eqref{eq-claim-lower} follows, and we define $v_{\varepsilon}:=u_{\varepsilon}-m_{\varepsilon}$ which is nonnegative in the neighborhood of $Q_{\varepsilon} \setminus B_{b^{\varepsilon}}$. Therefore, we are now able to apply Harnack's inequality on the circle $\partial B_{r_{\varepsilon}}$ (see \cite[Theorem 4.3]{CC95} for instance): 
	\begin{equation*}
		0 \leq v_{\varepsilon}(x) \leq Cv_{\varepsilon}(x^{\varepsilon}) \leq C\varepsilon \quad \text{for any $x\in \partial B_{r_{\varepsilon}}$}.
	\end{equation*}    
    We next let $G$ be the Green's function for $L$ in $B_{1/2}$.  Then in view of \Cref{lem-green1} and \Cref{lem-green2}, we consider an auxiliary function
	\begin{equation*}
		g_{\varepsilon}(x):=\sup_{\Omega} \varphi\cdot\frac{G(x, 0)-\min_{\partial B_{r_{\varepsilon}}}G(\cdot, 0)}{\min_{\partial B_{a^{\varepsilon}}}G(\cdot, 0)-\min_{\partial B_{r_{\varepsilon}}}G(\cdot, 0)}+C\varepsilon,
	\end{equation*}
	which satisfies $a_{ij}D_{ij}g_{\varepsilon}=0$ in $B_{r_{\varepsilon}}\setminus B_{a^{\varepsilon}}$ and $g_{\varepsilon} \geq v_{\varepsilon}$ on $\partial (B_{r_{\varepsilon}}\setminus B_{a^{\varepsilon}})$. Therefore, by the comparison principle, we obtain that for $x \in B_{r_{\varepsilon}}\setminus B_{b^{\varepsilon}}$,
	\begin{equation*}
		v_{\varepsilon} \leq g_{\varepsilon} \leq C(a^{\varepsilon})^{n-2} (b^{\varepsilon})^{2-n}+C\varepsilon \leq C({a}^{\varepsilon}/\varepsilon)^{\frac{n-2}{2}}+C\varepsilon.
	\end{equation*}
	Since $r_{\varepsilon} \in [\varepsilon/100, \varepsilon/2]$ and $v_{\varepsilon}$ is nonnegative, we employ the standard covering argument together with Harnack's inequality and the maximum principle (for finitely many times) to conclude the desired oscillation control in $Q_{\varepsilon} \setminus B_{b^{\varepsilon}}$.
\end{proof}	
	
We also need the following version of the Arzela-Ascoli theorem, whose proof is a simple modification of the original one. In other words, the equicontinuity assumption in the Arzela-Ascoli theorem can be relaxed to ``almost equicontinuity''.
\begin{lemma}[Arzela-Ascoli theorem] \label{lem-arzelaascoli}
	Suppose that a sequence of functions $\{f_m\}_{m \in \mathbb{N}}$ defined on $A (\subset \mathbb{R}^n)$ satisfies 
	\begin{enumerate}[(i)]
		\item(Uniformly bounded) There exists a constant $M>0$ such that $$|f_m(x)| \leq M,$$
		for any $m \in \mathbb{N}$ and $x \in A$.
		\item(Almost equicontinuous) There exist a constant $\alpha \in (0,1]$, $C>0$ and a function $g: \mathbb{N} \to \mathbb{R}_{\geq 0}$ such that $\lim_{m \to \infty}g(m)=0$ and
		\begin{align*}
			|f_m(x_1)-f_m(x_2)| \leq C|x_1-x_2|^{\alpha}+g(m),
		\end{align*}
		for any $x_1, x_2 \in A$.
	\end{enumerate}
	Then there exists a subsequence $\{f_{m_k}\}_{k \in \mathbb{N}}$ which converges uniformly on $A$. Moreover, if we denote the limit function by $f$, then $f \in C^{0, \alpha}(A)$.
\end{lemma}

We are now ready to describe a limit $u$ of $u_{\varepsilon}$ up to subsequence. We recall from \Cref{sec-introduction} that
	\begin{equation*}
		\underline{u}_{\varepsilon}(x)=
		\begin{cases}
			u_{\varepsilon}(x) & \text{if $x \in Q_{\varepsilon}(\varepsilon k) \setminus B_{b^{\varepsilon}}(\varepsilon k)$ for some $k \in \mathbb{Z}^n$}\\
			\min_{\partial B_{b^{\varepsilon}}(\varepsilon k)} u_{\varepsilon} & \text{if $x \in B_{b^{\varepsilon}}(\varepsilon k)$ for some $k \in \mathbb{Z}^n$}.
	\end{cases}
	\end{equation*}
 
\begin{theorem}[Convergence up to subsequence]\label{thm-convergence}
    Suppose that $\varepsilon_m \to 0$ as $m \to \infty$. Then there exists a function $u \in C^{0, 1}(\overline{\Omega})$ such that $\underline{u}_{\varepsilon_m} \to u$ uniformly in $\Omega$ up to subsequence. 
\end{theorem}

\begin{proof}
A combination of \Cref{lem-discrete} and \Cref{lem-cubeosc} yields that
\begin{equation*}
	|\underline{u}_{\varepsilon_m}(x_1)-\underline{u}_{\varepsilon_m}(x_2)| \leq C(|x_1-x_2|+\varepsilon_m+(a^{\varepsilon_m}/\varepsilon_m)^{\frac{n-2}{2}})
\end{equation*}
for any $x_1, x_2 \in \Omega$. Therefore, by applying \Cref{lem-arzelaascoli}, there exists a continuous function $u$ such that $\underline{u}_{\varepsilon_m} \to u$ uniformly in $\Omega$ up to subsequence.
\end{proof}
As mentioned in \Cref{sec-introduction}, we are going to characterize the effective equation satisfied by $u$ and to verify the uniform convergence of $\underline{u}_{\varepsilon}$ to $u$ in the full sequence in \Cref{sec-homogenization}.

\section{Correctors and critical values}\label{sec-correctors} 
    As mentioned in \Cref{sec-introduction}, we summarize the known results concerning $w_{1, D^2P}$, which corrects the oscillation of coefficient matrix $a_{ij}$. The following lemmas were shown in \cite{Eva89, Eva92}; see also \cite{KL16}.
\begin{lemma}\label{lem:periodic}
	For $M \in \mathcal{S}^n :=\{\text{$n \times n$ symmetric matrices} \}$, there exists a unique $\kappa \in \mathbb{R}$ for which the following equation admits a $1$-periodic solution 
	\begin{equation} \label{periodic}
		a_{ij}D_{y_iy_j}w+a_{ij}M_{ij}=\kappa \quad \text{in $\mathbb{R}^n$}.
	\end{equation}
	   Moreover, the solutions of \eqref{periodic} lie in $C^{2}(\mathbb{R}^n)$ and are unique up to an additive constant.
\end{lemma}
We further provide several properties for the functional $\kappa: \mathcal{S}^n \to \mathbb{R}$.
\begin{lemma}\label{lem:periodic2}
	Let $\kappa$ be the functional on $\mathcal{S}^n$ obtained from \Cref{lem:periodic}. 
	\begin{enumerate}[(i)]
		\item There is a constant symmetric matrix $(\overline{a}_{ij})$ such that $\kappa(M)=\overline{a}_{ij}M_{ij}$.
		
		\item The matrix $(\overline{a}_{ij})$ is elliptic with the same ellipticity constants of $(a_{ij})$; that is, 
		\begin{equation*}
			\lambda |\xi|^2 \leq \overline{a}_{ij}\xi_i\xi_j \leq \Lambda|\xi|^2 \quad \text{for all $\xi \in \mathbb{R}^n$}.
		\end{equation*}
	\end{enumerate}
\end{lemma}
By applying \Cref{lem:periodic}, we set the unique $1$-periodic solution $w_{1, M}$ which is normalized in the sense of
\begin{equation*}
	\min_{\mathbb{R}^n} w_{1, M}=0.
\end{equation*}
If there are no holes and corresponding oscillating obstacles, then the function $w_{1, M}$ is sufficient to explain the homogenization scheme. Even for our mixed setting, it is still valid for correcting the coefficient oscillation and determining the effective coefficient $\overline{a}_{ij}$. Therefore, in the remaining of this section, we concentrate on the construction of the corrector $w_2^{\varepsilon}$, which encodes the nontrivial effect of highly oscillating obstacles.

\subsection{Existence of periodic correctors}
We would like to determine the critical value $\beta^{\varepsilon}$ and the corresponding periodic corrector $w_2^{\varepsilon}$ (or $w^{\varepsilon}$ for simplicity) which denotes the corrector for the obstacle part introduced in \Cref{sec-introduction}. More precisely, the desirable pair $(\beta^{\varepsilon}, w^{\varepsilon})$ satisfies the following relation: 
	\begin{equation*}
			\left\{ \begin{aligned} 
					a_{ij}(x/\varepsilon)D_{ij}w^{\varepsilon}(x)&=\beta^{\varepsilon} && \text{in $\mathbb{R}^n \setminus T_{{a}^{\varepsilon}}$} \\
					w^{\varepsilon}&=1  && \text{on $\partial T_{{a}^{\varepsilon}}$}\\
					w^{\varepsilon}\text{ is $\varepsilon$-periodic} &\text{ and } \inf w^{\varepsilon}=0. &&
				\end{aligned} \right.
		\end{equation*}
	 Roughly speaking, the critical value $\beta^{\varepsilon}$ is uniquely chosen so that the oscillation of corrector $w^{\varepsilon}$ becomes exactly $1$. It is noteworthy that we will make use of correctors in different scales throughout this section to verify nice properties.
    
    In the remaining of the section, we will construct a $1$-periodic corrector $W^{\varepsilon}$ given by
	 \begin{equation*}
	 	W^{\varepsilon}(y):=\varepsilon^{-2}w^{\varepsilon}(\varepsilon y);
	 \end{equation*}
    we will scale back in \Cref{sec-homogenization}.
  We first concentrate on the periodicity of $W^{\varepsilon}$, which is based on the construction of periodic super/subsolutions. For simplicity, we define the periodic holes as follows:
  \begin{equation*}
      \overline{T}_{r}:=\bigcup_{k \in \mathbb{Z}^n} B_{r}(k) \quad \text{for $r>0$}.
  \end{equation*}

\begin{theorem}\label{thm-periodiccorrector}
	There exists a unique $1$-periodic solution $W^{\varepsilon}$ that solves
	\begin{equation}\label{eq-periodiccorrector}
		\left\{ \begin{aligned} 
			a_{ij}(y)D_{ij}W^{\varepsilon}(y)&=1 && \text{in $\mathbb{R}^n \setminus \overline{T}_{\overline{a}^{\varepsilon}}$} \\
			W^{\varepsilon}&=\varepsilon^{-2}  && \text{on $\partial \overline{T}_{\overline{a}^{\varepsilon}}$},
		\end{aligned} \right.
	\end{equation}
where $\overline{a}^{\varepsilon}:=a^{\varepsilon}/\varepsilon=\varepsilon^{\frac{2}{n-2}}$.
\end{theorem}

\begin{proof}
	The strong maximum principle implies that \eqref{eq-periodiccorrector} admits at most one solution. Moreover, in view of the Perron's method provided in \cite[Theorem 4.1]{CIL92}, it suffices to construct periodic super/subsolutions of \eqref{eq-periodiccorrector}. 
	
{\bf1. Supersolution $H^{\varepsilon, +}$}: We first denote by $G(y):=G(y, 0)$ the Green's function for the operator $a_{ij}D_{ij}$ in $B_{1/2}=B_{1/2}(0)$. We recall from \Cref{lem-green1} and \Cref{lem-green2} that
\begin{equation*}
	\left\{ \begin{aligned} 
		a_{ij}(y)D_{ij}G(y)&=0 && \text{in $B_{1/2} \setminus \{0\}$} \\
		G(y)&=0  && \text{on $\partial B_{1/2}$},
	\end{aligned} \right.
\end{equation*}
and that there exist universal constants $c \geq 1$ and $r_0>0$ such that
\begin{equation*}
	\frac{1}{c}|y|^{2-n} \leq G(y) \leq c|y|^{2-n} \quad \text{when $0<|y| \leq r_0$.}
\end{equation*}
We next let $h^{\varepsilon}$ be the viscosity solution of 
			\begin{equation*}
			\left\{ \begin{aligned} 
				a_{ij}(y)D_{ij}h^{\varepsilon}(y)&=0 && \text{in $B_{1/2} \setminus B_{\overline{a}^{\varepsilon}}$} \\
				h^{\varepsilon}&=\varepsilon^{-2}  && \text{on $\partial B_{\overline{a}^{\varepsilon}}$}\\
				h^{\varepsilon}&=0  && \text{on $\partial B_{1/2}$}.
			\end{aligned} \right.
		\end{equation*}
		Without loss of generality, choose $\varepsilon>0$ small enough so that $\overline{a}^{\varepsilon} \leq \min\{r_0, 1/4\}$.  Then by applying the comparison principle on $B_{1/2} \setminus B_{\overline{a}^{\varepsilon}}$, we have
		\begin{equation}\label{eq-uniformestimate}
			c^{-1}G(y) \leq h^{\varepsilon}(y) \leq cG(y) \quad \text{in $B_{1/2} \setminus  B_{\overline{a}^{\varepsilon}}$}.
		\end{equation}
	In particular, the estimate \eqref{eq-uniformestimate} implies that there exists a universal constant $M>0$ (which is independent of $\varepsilon>0$) such that
		\begin{equation*}
			-M \leq \partial_{\nu}h^{\varepsilon}\leq 0 \quad \text{on $\partial B_{1/2}$},
		\end{equation*}
		where $\nu$ is the outer normal unit vector for $B_{1/2}$. 
		
		We now define $H^{\varepsilon, +}_0: Q \setminus B_{\overline{a}^{\varepsilon}}\to \mathbb{R}$ by
		\begin{equation*}
			H^{\varepsilon, +}_0(y):=
			\left\{ \begin{array}{ll} 
				h^{\varepsilon}+2M|y|^2 & \text{in $B_{1/2}\setminus B_{\overline{a}^{\varepsilon}}$} \\
				M/2  & \text{in $Q \setminus B_{1/2}$},
			\end{array} \right.
		\end{equation*}
		where $Q=Q_1(0)=[-1/2, 1/2]^n$. Then by a direct computation, we observe that $H_0^{\varepsilon, +}$ is a viscosity supersolution of
		\begin{equation*}
			a_{ij}(y)D_{ij}H_0^{\varepsilon, +} \leq 4Mn\Lambda,
		\end{equation*}
		 in the interior of $Q \setminus B_{1/2}$ and $B_{1/2} \setminus B_{\overline{a}^{\varepsilon}}$. It only remains to check a supersolution property on $\partial B_{1/2}$. Indeed, for any $y_0 \in \partial B_{1/2}$,
			\begin{equation*}
			\left\{ \begin{array}{ll} 
				\lim_{t \searrow 0}\frac{H_0^{\varepsilon, +}(y_0-t\nu)-H_0^{\varepsilon, +}(y_0)}{t}=\partial_{\nu}h^{\varepsilon}+2M \geq M &  \\
				\lim_{t \searrow 0}\frac{H_0^{\varepsilon, +}(y_0+t\nu)-H_0^{\varepsilon, +}(y_0)}{t}=0,  &
			\end{array} \right.
		\end{equation*}
	where $\nu$ is the outer normal unit vector for $B_{1/2}$. In other words, a quadratic polynomial cannot touch $H_0^{\varepsilon, +}$ from below on $\partial B_{1/2}$.
	
	Therefore, we conclude that $H_0^{\varepsilon, +}$ is a viscosity supersolution of 
	\begin{equation}\label{eq-supersolution}
	\left\{ \begin{aligned} 
		a_{ij}(y)D_{ij}H_0^{\varepsilon, +}(y) &\leq 4Mn\Lambda && \text{in $Q \setminus B_{\overline{a}^{\varepsilon}}$} \\
		H_0^{\varepsilon, +}&\geq \varepsilon^{-2}  && \text{on $\partial B_{\overline{a}^{\varepsilon}}$}.
	\end{aligned} \right.
\end{equation}
	We finally construct a periodic supersolution $H^{\varepsilon, +}$ by
	\begin{equation*}
		H^{\varepsilon, +}(y):=
			\left\{ \begin{array}{ll} 
			H_0^{\varepsilon, +}(y-k)  & \text{if $y-k \in Q$ for $k \in \mathbb{Z}^n$} \\
			M/2  & \text{if $y-k \in \partial Q$ for $k \in \mathbb{Z}^n$}.
		\end{array} \right.
	\end{equation*}
	Then by a similar argument as before, $H^{\varepsilon, +}$ becomes a $1$-periodic viscosity supersolution of \eqref{eq-supersolution} in $\mathbb{R}^n \setminus \overline{T}_{\overline{a}^{\varepsilon}}$. The desired supersolution to \eqref{eq-periodiccorrector} can be obtained by the standard scaling and translation.

{\bf2. Subsolution $H^{\varepsilon, -}$}: 

2-1. Definition of $\widetilde{h}^{\varepsilon}$: Similarly as in the supersolution case, we denote by $\widetilde{G}(y):=G(y, 0)$ the Green's function for the operator $a_{ij}D_{ij}$ in $B_{\sqrt{n}}$. We note that $Q \subset\subset B_{\sqrt{n}}$. We again recall from \Cref{lem-green1} and \Cref{lem-green2} that
\begin{equation*}
	\left\{ \begin{aligned} 
		a_{ij}(y)D_{ij}\widetilde{G}(y)&=0 && \text{in $B_{\sqrt{n}} \setminus \{0\}$} \\
		\widetilde{G}(y)&=0  && \text{on $\partial B_{\sqrt{n}}$},
	\end{aligned} \right.
\end{equation*}
and that there exist universal constants $c \geq 1$ and $r_0>0$ such that
\begin{equation*}
	\frac{1}{c}|y|^{2-n} \leq \widetilde{G}(y) \leq c|y|^{2-n} \quad \text{when $0<|y| \leq r_0$.}
\end{equation*}
We next let $\widetilde{h}^{\varepsilon}$ be the viscosity solution of 
\begin{equation*}
	\left\{ \begin{aligned} 
		a_{ij}(y)D_{ij}\widetilde{h}^{\varepsilon}(y)&=0 && \text{in $B_{\sqrt{n}} \setminus B_{\overline{a}^{\varepsilon}}$} \\
		\widetilde{h}^{\varepsilon}&=\varepsilon^{-2}  && \text{on $\partial B_{\overline{a}^{\varepsilon}}$}\\
		\widetilde{h}^{\varepsilon}&=0  && \text{on $\partial B_{\sqrt{n}}$}.
	\end{aligned} \right.
\end{equation*}
Again by applying the comparison principle on $B_{\sqrt{n}} \setminus B_{\overline{a}^{\varepsilon}}$, we have
\begin{equation}\label{eq-uniformestimate2}
	c^{-1}\widetilde{G}(y) \leq \widetilde{h}^{\varepsilon}(y) \leq c\widetilde{G}(y) \quad \text{in $B_{\sqrt{n}} \setminus  B_{\overline{a}^{\varepsilon}}$}.
\end{equation}
In particular, the estimate \eqref{eq-uniformestimate2} together with the Hopf maximum principle for $\widetilde{G}$ shows that there exists a universal constant $\widetilde{M}>0$ (which is independent of $\varepsilon>0$) such that
\begin{equation}\label{eq-hopf1}
	\partial_{\nu}\widetilde{h}^{\varepsilon}\leq -\widetilde{M} \quad \text{on $\partial B_{\sqrt{n}}$}.
\end{equation}
On the other hand, we define $p(y):=(|y|^2-n)/(2n\lambda)$ which enjoys
\begin{equation*}
	\left\{ \begin{aligned} 
		a_{ij}(y)D_{ij}p(y) &\geq 1 && \text{in $B_{\sqrt{n}} \setminus B_{\overline{a}^{\varepsilon}}$} \\
		p &<0 && \text{in $B_{\sqrt{n}}$}\\
		p&=0  && \text{on $\partial B_{\sqrt{n}}$}\\
		\partial_{\nu}p&=(\sqrt{n}\lambda)^{-1} && \text{on $\partial B_{\sqrt{n}}$}.
	\end{aligned} \right.
\end{equation*}

2-2. Non-degeneracy of $\widetilde{h}^{\varepsilon}$ in terms of $p$: We claim that there exists a universal constant $\delta>0$ such that 
\begin{equation*}
	\widetilde{h}^{\varepsilon}+\delta p>0 \quad \text{in $B_{\sqrt{n}} \setminus B_{\overline{a}^{\varepsilon}}$}.
\end{equation*}
Suppose not; that is, for any $m \in \mathbb{N}$, there exist sequences  $\{y_m\}_{m \in \mathbb{N}} \subset B_{\sqrt{n}} \setminus B_{1/4}$ and $\{\varepsilon_m\}_{m \in \mathbb{N}} \subset (0, 1)$ such that
\begin{equation}\label{eq-contradiction}
	\widetilde{h}^{\varepsilon_m}(y_m) \leq -\frac{1}{m}p(y_m).
\end{equation}
Then there exists a subsequence of $\{y_m\}$ such that $y_m \to y_{\infty} \in \overline{B_{\sqrt{n}} \setminus B_{1/4}}$. Since $p(y_m)/m$ goes to zero when $m$ tends to $\infty$, it follows that $y_{\infty} \in \partial B_{\sqrt{n}}$.

On the other hand, we choose $z_m \in \partial B_{\sqrt{n}}$ so that $\mathrm{dist}(y_m, \partial B_{\sqrt{n}})=|y_m-z_m|>0$. Since $y_m \to y_{\infty} \in \partial B_{\sqrt{n}}$, such choice of $z_m$ is unique whenever $m$ is sufficiently large. Then by applying Taylor's theorem together with \eqref{eq-hopf1}, it turns out that 
\begin{equation*}
	\widetilde{h}^{\varepsilon_m}(y_m) \geq \widetilde{h}^{\varepsilon_m}(z_m)+\frac{1}{2}D\widetilde{h}^{\varepsilon_m}(z_m)\cdot(y_m-z_m) \geq \frac{\widetilde{M}}{2}|y_m-z_m|.
\end{equation*}
Similarly, we obtain
\begin{equation*}
	-p(y_m) \leq \frac{1}{2\sqrt{n}\lambda}|y_m-z_m|.
\end{equation*}
Therefore, by taking these estimates into \eqref{eq-contradiction}, we arrive at
\begin{equation*}
	\sqrt{n}\lambda\widetilde{M} \leq \frac{1}{m},
\end{equation*}
which is impossible.

2-3. Construction of $H^{\varepsilon, -}$: We now choose $\delta>0$ found in the step 2-2 and define $H_0^{\varepsilon, -}: \mathbb{R}^n \setminus B_{\overline{a}^{\varepsilon}} \to \mathbb{R}$ by
\begin{equation*}
	H_0^{\varepsilon, -}(y):=
	\left\{ \begin{array}{ll} 
		\widetilde{h}^{\varepsilon}+\delta p & \text{in $B_{\sqrt{n}} \setminus B_{\overline{a}^{\varepsilon}}$}\\
		0 & \text{otherwise}.
	\end{array} \right.
\end{equation*}
Then we observe that
\begin{equation*}
	\left\{ \begin{aligned} 
		a_{ij}(y)D_{ij}H_0^{\varepsilon, -}(y) &\geq \delta && \text{in $B_{\sqrt{n}} \setminus B_{\overline{a}^{\varepsilon}}$} \\
		H_0^{\varepsilon, -}& \leq \varepsilon^{-2}  && \text{on $\partial B_{\overline{a}^{\varepsilon}}$}\\
		H_0^{\varepsilon, -}&=0  && \text{on $\mathbb{R}^n \setminus B_{\sqrt{n}}$}\\
		H_0^{\varepsilon, -}&>0  && \text{on $B_{\sqrt{n}} \setminus B_{\overline{a}^{\varepsilon}}$}.
	\end{aligned} \right.
\end{equation*}
We finally define a $1$-periodic function
\begin{equation*}
	H^{\varepsilon, -}(y)=\sup_{k \in \mathbb{Z}^n} H_0^{\varepsilon, -}(y-k) \quad \text{for $y \in \mathbb{R}^n \setminus \overline{T}_{\overline{a}^{\varepsilon}}$.}
\end{equation*}
By step 2-2, the supremum can be taken over $k \in \mathbb{Z}^n$ such that $y-k \in B_{\sqrt{n}}$. Then $H^{\varepsilon, -}$ is a viscosity subsolution of
\begin{equation*}
	\left\{ \begin{aligned} 
		a_{ij}(y)D_{ij}H^{\varepsilon, -}(y)&\geq \delta && \text{in $\mathbb{R}^n \setminus \overline{T}_{\overline{a}^{\varepsilon}}$} \\
		H^{\varepsilon, -}&\leq \varepsilon^{-2}  && \text{on  $\partial \overline{T}_{\overline{a}^{\varepsilon}}$},
	\end{aligned} \right.
\end{equation*}
because
\begin{enumerate}[(i)]
	\item if $y-k \in B_{\sqrt{n}}$ for unique $k \in \mathbb{Z}^n$, then $H^{\varepsilon, -}(\cdot)=H_0^{\varepsilon, -}(\cdot-k)$ near the point $y$;
	
	\item  if $y-k \in B_{\sqrt{n}}$ for different $k \in \mathbb{Z}^n$, then we note that: if $u_1$ and $u_2$ are viscosity subsolutions in $\Omega$, then so is $\sup\{u_1, u_2\}$.
\end{enumerate}
The desired subsolution to \eqref{eq-periodiccorrector} can be obtained by standard scaling and translation.
\end{proof}

\begin{corollary}\label{cor-critical}
		Let $\varepsilon \in (0,1)$. Then there exists a unique constant $\beta^{\varepsilon} > 0$ such that
			\begin{equation*}
			\left\{ \begin{aligned} 
				a_{ij}(y)D_{ij}W^{\varepsilon}(y)&=\beta^{\varepsilon} && \text{in $\mathbb{R}^n \setminus \overline{T}_{\overline{a}^{\varepsilon}}$} \\
				W^{\varepsilon}&=\varepsilon^{-2}  && \text{on $\partial \overline{T}_{\overline{a}^{\varepsilon}}$}\\
				\inf W^{\varepsilon}&=0  && 
			\end{aligned} \right.
		\end{equation*}
		admits a unique $1$-periodic solution $W^{\varepsilon}$. Moreover, there exists a constant $C>0$ depending only on $n$, $\lambda$, and $\Lambda$ such that
		\begin{equation*}
			0 < \beta^{\varepsilon} \leq C.
		\end{equation*}
\end{corollary}

\begin{proof}
	For $\varepsilon>0$, let $W^{\varepsilon}$ be the unique $1$-periodic solution given by \Cref{thm-periodiccorrector}. Then we set 
	\begin{equation*}
			W^{\varepsilon}_{\eta}:=\eta \varepsilon^{-2}+(1-\eta) W^{\varepsilon}
		\end{equation*}
	 for $\eta$ to be determined. Then $W^{\varepsilon}_{\eta}$ solves
	 \begin{equation*}
		 	\left\{ \begin{aligned} 
			 		a_{ij}(y)D_{ij}W^{\varepsilon}_{\eta}(y)&=1-\eta && \text{in $\mathbb{R}^n \setminus \overline{T}_{\overline{a}^{\varepsilon}}$} \\
			 		W_{\eta}^{\varepsilon}&=\varepsilon^{-2}  && \text{on $\partial \overline{T}_{\overline{a}^{\varepsilon}}$}\\
			 		\inf W_{\eta}^{\varepsilon}&=\eta \varepsilon^{-2}+(1-\eta)\inf W^{\varepsilon}  && 
			 	\end{aligned} \right.
		 \end{equation*}
 	Since $\inf W^{\varepsilon} <\varepsilon^{-2}$ due to the strong maximum principle, there exists a unique $\eta \, (<1)$ such that $\inf W^{\varepsilon}_{\eta}=0$. Moreover, the estimate $\beta^{\varepsilon}\leq C$ immediately follows from the construction of supersolutions in the proof of \Cref{thm-periodiccorrector}.
\end{proof}

In fact, we can say more about the behaviors of periodic correctors $W^{\varepsilon}$.
\begin{lemma}\label{lem-correctorbound}
Let $(\beta^{\varepsilon}, W^{\varepsilon})$ be the unique pair found in \Cref{cor-critical}. Then the following hold:
\begin{enumerate}[(i)]
    \item The infimum of $W^{\varepsilon}$ is attained in $Q_1 \setminus B_{r_1}$ for some universal constant $r_1 \in (0, 1/2)$; that is, there exists some universal constant $r_1>0$ such that $W^\varepsilon(y) > 0$ if $\overline{a}^{\varepsilon} \leq |y| \le r_1$.

In particular, there exists a universal constant $C>0$ such that
\begin{equation*}
    0 \leq W^\varepsilon \le C \quad \text{in $\mathbb R^n \setminus \overline{T}_{r_1}$.}
\end{equation*}

    \item There exist universal constants $C_1, C_2>0$ such that
    \begin{equation*}
       C_1|y|^{2-n} \leq  W^{\varepsilon}(y) \leq C_2|y|^{2-n} \quad \text{if $\overline{a}^{\varepsilon} \leq |y| \leq r_1$}.
    \end{equation*}

\end{enumerate}
\end{lemma}
\begin{proof}
(i) Let $h^\varepsilon$ be a solution of
\begin{equation*}
	\left\{ \begin{aligned} 
		a_{ij}(y) D_{ij} h^\varepsilon &= \beta^\varepsilon && \text{in } B_{1/2} \\
		h^\varepsilon &= 0  && \text{on } \partial B_{1/2}.
	\end{aligned} \right.
\end{equation*}
By the maximum principle, there exists a universal constant $C>0$ such that
\begin{equation*}
 -C \leq h^{\varepsilon} \leq 0 \quad \text{in $B_{1/2}$}.
\end{equation*}
Let $\widetilde W^\varepsilon = W^\varepsilon - h^\varepsilon$. Then, for $\varepsilon>0$ small enough, $\widetilde W^\varepsilon$ satisfies
\begin{equation}
	\left\{ \begin{aligned} 
		a_{ij} D_{ij} \widetilde W^\varepsilon &= 0 &&\text{in } B_{1/2} \setminus B_{\overline a^\varepsilon}\\
\widetilde W^\varepsilon &\geq 1/(2\varepsilon^2) &&\text{on } \partial B_{\overline a^\varepsilon}\\ 
\widetilde W^\varepsilon &\ge 0 &&\text{on } \partial B_{1/2}.
	\end{aligned} \right.
\end{equation}
From \eqref{eq-almost}, there exist positive constants $c_1$, $c_2$, and $r_0$ such that 
\begin{equation*}
c_1|y|^{2-n} \le G(y,0) \le c_2|y|^{2-n} \quad \text{if $|y|\leq r_0$},
\end{equation*}
where $G$ is the Green's function for $L$ in $B_{1/2}$. In particular, 
\begin{equation*}
c_1\varepsilon^{-2} \le G(y,0) \le c_2\varepsilon^{-2} 
\end{equation*}
holds on $\partial B_{\overline a^\varepsilon}$ for sufficiently small $\varepsilon$. Since $G(y,0)=0$ on $\partial B_{1/2}$, we have from comparison that 
\begin{equation*}
\frac{1}{2c_2}G(y,0) \le \widetilde W^\varepsilon \quad \text{in $B_{1/2} \setminus B_{\overline a^\varepsilon}$}
\end{equation*}
and so
\begin{equation*}
\frac{|y|^{2-n}}{2c_2} \le W^\varepsilon+C \quad \text{in $B_{r_0} \setminus B_{\overline a^\varepsilon}$}.
\end{equation*}
Hence, if we take $r_1 = \min \{r_0, (4Cc_2)^{1/(2-n)}\}$, then $W^\varepsilon \ge C>0$ in $B_{r_1} \setminus B_{\overline a^\varepsilon}$, which proves the first part of the lemma.

On the other hand, the Harnack inequality for $W^{\varepsilon}$ in $Q\setminus B_{r_1}$ tells us that 
\begin{equation*}
\sup_{Q \setminus B_{r_1}} W^\varepsilon \le C\inf_{Q \setminus B_{r_1}} W^\varepsilon+C\beta^{\varepsilon} \leq C.
\end{equation*}

(ii) It follows from a similar argument as in (i).
\end{proof}

\subsection{Uniqueness of the critical value \texorpdfstring{$\beta_0$}{b0}}
We now move our attention to a limit of the critical value $\beta^{\varepsilon}$. Indeed, we are going to prove the uniqueness of such a limit, by investigating the asymptotic behavior of $W^{\varepsilon}$ near the singularity $k \in \mathbb{Z}^n$. It is noteworthy that \Cref{thm-uniquebeta} can be developed in an alternative way exploiting the notion of capacity, when it comes to the simplest case of $a_{ij}(\cdot)=\delta_{ij}$; see \cite[Lemma 2.1]{CL08}.

We consider another scaled function of $w^{\varepsilon}$ and $W^{\varepsilon}$:
\begin{equation*}
    \overline{W}^\varepsilon(z) := \varepsilon^2 W^\varepsilon(\overline a^\varepsilon z).
\end{equation*}
Then we observe that $\overline W^\varepsilon$ satisfies
\begin{equation*}
\left\{ \begin{aligned} 
a_{ij}(\overline a^\varepsilon z) D_{ij} \overline W^\varepsilon(z) &= (a^\varepsilon)^2 \beta^\varepsilon && \text{in $B_{1 / (2\overline a^\varepsilon)} \setminus B_1$} \\
\overline{W}^{\varepsilon}&=1  && \text{on $\partial B_1$}.
\end{aligned} \right.
\end{equation*}
Moreover, we denote by $\overline{W}_0$ the $L_0$-capacity potential  of $B_1$, where $L_0=a_{ij}(0)D_{ij}$. More precisely, $\overline{W}_0$ is the unique solution of 
 \begin{equation*}
		 	\left\{ \begin{aligned}
			 		a_{ij}(0)D_{ij}\overline{W}_0&=0 && \text{in $\mathbb{R}^n \setminus B_1$} \\
			 		\overline{W}_0&=1  && \text{on $\partial B_1$}\\
			 		\overline{W}_0(z) &\to 0  && \text{as $|z| \to \infty$}. 
			 	\end{aligned} \right.
		 \end{equation*}
   Then it is to check that there exists a constant $C=C(n, \lambda, \Lambda)>0$ such that
   \begin{equation}\label{eq-cappotential}
       0 \leq \overline{W}_0(z) \leq C|z|^{2-n} \quad \text{in $\mathbb{R}^n \setminus B_1$}
   \end{equation}
   and
   \begin{equation}\label{eq-cappotential2}
       |D^2\overline{W}_0(z)| \leq C|z|^{-n} \quad \text{in $\mathbb{R}^n \setminus B_1$}.
   \end{equation}

\begin{lemma} \label{lem-we-difference}
Let $r_1 \in (0, 1/2)$ be the constant chosen in \Cref{lem-correctorbound}. If $R  \in [1, r_1/\overline{a}^{\varepsilon}]$ and $p \in [1, \infty)$, then we have
\begin{equation*}
\| \overline W^\varepsilon -\overline W_0 \|_{L^{\infty}(B_R \setminus B_1)}+\| \overline{W}^\varepsilon - \overline{W}_0 \|_{W^{2,p}(B_R\setminus B_1)} \le C \left(\varepsilon^2 + R^{2-n}+\phi(\overline{a}^{\varepsilon} R)\, R^2\right),
\end{equation*}
where $\phi$ denotes the modulus of continuity of $a_{ij}$, i.e., $\phi:[0, \infty) \to [0, \infty)$ satisfies $\phi(0+)=0$ and 
\begin{equation*}
    |a_{ij}(z_1)-a_{ij}(z_2)|\leq \phi(|z_1-z_2|) \quad \text{for any $z_1, z_2 \in \mathbb{R}^n$.}
\end{equation*}
\end{lemma}

\begin{proof}
We first note that the existence of $\phi$ is guaranteed by \Cref{rmk-dini}. Then by applying \Cref{lem-correctorbound} (ii) and \eqref{eq-cappotential}, we observe that 
\begin{equation*}
|\overline{W}^\varepsilon(z) - \overline{W}_0(z)| \le |\overline{W}^\varepsilon(z)| + |\overline{W}_0(z)| \le C\left( \varepsilon^2 + R^{2-n} \right)
\end{equation*}
holds for $|z|=R$, provided that $\overline{a}^{\varepsilon} R \leq r_1$. 

Moreover, $\overline{W}^\varepsilon - \overline{W}_0$ satisfies
\begin{equation*}
a_{ij}(\overline a^\varepsilon z) D_{ij}(\overline W^\varepsilon -\overline W_0) = (a_{ij}(\overline a^\varepsilon z) - a_{ij}(0))D_{ij} \overline{W}_0+ (a^\varepsilon)^2 \beta^\varepsilon \quad \text{ in $B_R\setminus B_1$}.
\end{equation*}
Thus, it follows from the maximum principle that 
\begin{equation*}
\| \overline W^\varepsilon -\overline W_0 \|_{L^{\infty}(B_R \setminus B_1)} \le C\left[ \left( \varepsilon^2 +R^{2-n} \right) +R^2\left(\phi(\overline{a}^{\varepsilon}R) + (a^\varepsilon)^2\right)  \right],
\end{equation*}
where we used \Cref{cor-critical} and \eqref{eq-cappotential2}. Since 
$(a^{\varepsilon}R)^2=\varepsilon^2 (\overline{a}^{\varepsilon}R)^2 \leq r_1^2\varepsilon^2$, an application of the $L^p$-theory yields the desired estimate.
\end{proof}

\begin{theorem}\label{thm-uniquebeta}
Let $\beta^{\varepsilon}$ be the critical value determined by \Cref{cor-critical}; that is, 	
\begin{equation*}
\left\{ \begin{aligned} 
a_{ij}(y)D_{ij}W^{\varepsilon}(y)&=\beta^{\varepsilon} && \text{in $\mathbb{R}^n \setminus \overline{T}_{\overline{a}^{\varepsilon}}$} \\
W^{\varepsilon}&=\varepsilon^{-2}  && \text{on $\partial \overline{T}_{\overline{a}^{\varepsilon}}$}\\
\inf W^{\varepsilon}&=0  && 
\end{aligned} \right.
\end{equation*}
admits a unique $1$-periodic solution $W^{\varepsilon}$. Then there exists a unique limit $\beta$ of $\beta^{\varepsilon}$.
\end{theorem}

\begin{proof}
(i) Extensions and perturbations in the variable $z=y/\overline{a}^{\varepsilon}$: Let $\Gamma_0$ be the fundamental solution of $L_0=a_{ij}(0)D_{ij}$. We note that $\Gamma_0$ is homogeneous, that is, $\Gamma_0(tz)=t^{2-n}\Gamma_0(z)$ for any $t > 0$ and $z \in \mathbb{R}^n \setminus \{0\}$. Then we define
\begin{equation*}
    \overline{V}_0(z):=\overline{W}_0(z)-\gamma_0\Gamma_0(z)=\overline{W}_0(z)-\gamma_0\varepsilon^2\Gamma_0(\overline{a}^{\varepsilon}z),
\end{equation*}
where $\overline{W}_0$ is the $L_0$-capacity potential of $B_1$ as before and 
\begin{equation*}
    \gamma_0:=\mathrm{cap}_{L_0}(B_1)=-\int_{\partial B_1} a_{ij}(0)D_i\overline{W}_0 \nu_j \,d\sigma_z 
\end{equation*}
Here $\nu$ denotes the outer unit normal vector. Since $\overline{V}_0$ belongs to $C^{\infty}(\mathbb{R}^n \setminus B_1)$, we consider an extension $\overline{V}_0 \in C^2(\mathbb{R}^n)$. Such an extension is not unique, but we just choose one of them and call it $\overline{V}_0$. If we further define 
\begin{equation*}
   \overline{F}_0(z):=a_{ij}(0)D_{ij}\overline{V}_0,
\end{equation*}
then we observe that for $R \geq 1$,
\begin{equation} \label{f0-vanish}
\begin{aligned}
    \int_{B_R} \overline{F}_0 \,dz=\int_{B_1} \overline{F}_0 \,dz&= \int_{\partial B_1} a_{ij}(0) D_i \overline{V}_0 \nu_j\,d\sigma_z\\
    &=\int_{\partial B_1} a_{ij}(0) D_i \overline{W}_0 \nu_j\,d\sigma_z-\gamma_0\int_{\partial B_1} a_{ij}(0) D_i \Gamma_0 \nu_j\,d\sigma_z=0.
\end{aligned}
\end{equation}

Furthermore, we let $G(\cdot)=G(\cdot, 0)$ be the Green's function for $L=a_{ij}(\cdot)D_{ij}$ in $B_{1/2}$. We also define
\begin{equation*}
    \overline{H}^{\varepsilon}(z):=\varepsilon^2 G(\overline{a}^{\varepsilon}z)
\end{equation*}
and
\begin{equation*}
\overline{V}^{\varepsilon} (z) := \eta(z) (\overline W^\varepsilon (z) - \gamma_0 \overline{H}^\varepsilon(z)) + (1-\eta(z)) \overline{V}_0(z),
\end{equation*}
where $\eta\in C_c^{\infty}(\mathbb{R}^n)$ is a cut-off function such that $0 \leq \eta \leq 1$, $\eta =0$ in $B_1$, and $\eta=1$ in $\mathbb{R}^n \setminus B_2$. Then the function $\overline{V}^{\varepsilon}$ is well-defined in $B_{1/(2\overline{a}^{\varepsilon})}$. Moreover, 
since
\begin{equation*}
    \overline{V}^\varepsilon - \overline{V}_0 = \eta [ (\overline W^\varepsilon - \overline W_0) - \gamma_0 (\overline{H}^\varepsilon -\Gamma_0)],
\end{equation*}
we have for any $p \geq 1$, 
\begin{equation} \label{eq-vepsilon-difference}
	\| \overline{V}^\varepsilon - \overline{V}_0\|_{W^{2,p}(B_2)} \le C \{\| \overline W^\varepsilon - \overline W_0 \|_{W^{2,p}(B_2\setminus B_1)} + \|\overline{H}^\varepsilon -\Gamma_0 \|_{W^{2,p}(B_2\setminus B_1)}\}.
\end{equation}
We recall from \Cref{lem-we-difference} that if $R  \in [1, r_1/\overline{a}^{\varepsilon}]$, then
\begin{equation*}
\| \overline{W}^\varepsilon - \overline{W}_0 \|_{W^{2,p}(B_R\setminus B_1)} \le C \left(\varepsilon^2 + R^{2-n}+\phi(\overline{a}^{\varepsilon} R)\, R^2\right),
\end{equation*}
where $\phi(0+)=0$. In particular, we choose $R=R_{\varepsilon}>0$ so that $R_{\varepsilon}^4\cdot\phi(\overline{a}^{\varepsilon}R_{\varepsilon})=1$. Then it is easy to check that 
\begin{equation*}
    \lim_{\varepsilon \to 0}R_{\varepsilon}=\infty \quad \text{and} \quad \lim_{\varepsilon \to 0}\overline{a}^{\varepsilon}R_{\varepsilon}=0.
\end{equation*}
Therefore, it follows that
\begin{equation*}
\| \overline{W}^\varepsilon - \overline{W}_0 \|_{W^{2,p}(B_2\setminus B_1)} \leq \| \overline{W}^\varepsilon - \overline{W}_0 \|_{W^{2,p}(B_{R_{\varepsilon}}\setminus B_1)}\le C \left(\varepsilon^2 + R_{\varepsilon}^{2-n}+ R_{\varepsilon}^{-2}\right) \to 0
\end{equation*}
when $\varepsilon \to 0$.

On the other hand, by applying \Cref{lem-green2} and following the similar argument as in the proof of \Cref{lem-we-difference}, we obtain that 
\begin{equation*}
\| \overline{H}^{\varepsilon} -\Gamma_0  \|_{W^{2,p}(B_2\setminus B_1)} \to 0 \quad \text{when $\varepsilon \to 0$.}
\end{equation*}
By combining these estimates, we conclude that
\begin{equation}\label{eq-estimate-fe}
    \| \overline{V}^\varepsilon - \overline{V}_0\|_{W^{2,p}(B_2)}+\| \overline{F}^\varepsilon - \overline{F}_0\|_{L^p(B_2)} \to 0 \quad \text{when $\varepsilon \to 0$},
\end{equation}
where
\begin{equation*}
    \overline{F}^{\varepsilon}(z):=a_{ij}(\overline{a}^{\varepsilon}z) D_{ij}\overline{V}^{\varepsilon}(z).
\end{equation*}

(ii) Decomposition and convergence in the fast variable $y$: We now scale back to the fast variable $y$. Let $V^\varepsilon(y) = \varepsilon^{-2} \overline{V}^\varepsilon(y/\overline a^\varepsilon)$ and $F^\varepsilon(y) = (a^\varepsilon)^{-2} \overline{F}^\varepsilon(y/\overline a^\varepsilon) $ so that
\begin{equation*}
    a_{ij}(y)D_{ij}V^{\varepsilon}(y)=F^{\varepsilon}(y) \quad \text{for $y \in B_{1/2}$}.
\end{equation*}
It is easy to check that if $|z|>2$, then $\overline{V}^{\varepsilon}(z)=\overline{W}^{\varepsilon}(z)-\gamma_0 \varepsilon^2G(\overline{a}^{\varepsilon}z)$ and so $\overline{F}^{\varepsilon}(z)=(a^{\varepsilon})^2\beta^{\varepsilon}$. In other words, we have
\begin{equation*}
    F^{\varepsilon}(y)=\beta^{\varepsilon} \quad \text{for $|y|>2\overline{a}^{\varepsilon}$}.
\end{equation*}
Then by letting $G$ be the Green's function for $L$ in $B_{1/2}$ and recalling \Cref{rmk-green}, we can decompose $V^\varepsilon$ into two parts as $V^\varepsilon = V_1^\varepsilon + V_2^\varepsilon$, where  
\begin{equation*}
V_1^\varepsilon(y) :=  -\int_{B_{2\overline{a}^{\varepsilon}}} G(y, \tilde y) ( F^\varepsilon(\tilde y) -\beta^\varepsilon )  \, d\tilde y
\end{equation*}
and $V_2^{\varepsilon}$ is the unique solution of the Dirichlet problem
\begin{equation*}
\left\{ \begin{aligned} 
a_{ij}(y) D_{ij} V_2^\varepsilon &= \beta^{\varepsilon} && \text{in } B_{1/2} \\
V_2^\varepsilon &= V^\varepsilon = W^\varepsilon - \gamma_0 G(\cdot,0)  &&\text{on } \partial B_{1/2}.
\end{aligned} \right.
\end{equation*}

Let us now investigate limit behaviors of $V_i^{\varepsilon}$ when $\varepsilon \to 0$. In fact, since $W^\varepsilon$ has a uniform bound on every compact subset of $\mathbb R^n \setminus \mathbb Z^n$ due to \Cref{lem-correctorbound}, there exist a pair $(W, \beta)$ and a subsequence $\{\varepsilon_l\}_{l \in \mathbb{N}} \subset (0,1)$ (converging to $0$ when $l \to \infty$) such that 
\begin{equation*}
W^{\varepsilon_l} \to W\text{ locally uniformly in } \mathbb R^n \setminus \mathbb Z^n \text{ and } \beta^{\varepsilon_l} \to \beta \quad \text{when $l \to \infty$}.
\end{equation*}
First of all, due to the convergence of $W^{\varepsilon_l}$ and $\beta^{\varepsilon_l}$, it turns out that $V_2^{\varepsilon_l}$ converges uniformly in $B_{1/2}$ to $V_2 \in C(\overline{B_{1/2}})$, which solves 
\begin{equation*}
\left\{ \begin{aligned}
    a_{ij}(y) D_{ij} V_2 &=\beta &&\text{in } B_{1/2}\\
V_2 &= W - \gamma_0 G(\cdot,0) &&\text{on } \partial B_{1/2}.
\end{aligned}\right.
\end{equation*}
Moreover, by recalling \Cref{thm-L1data} and \eqref{eq-estimate-fe}, we observe that
\begin{equation*}
    \begin{aligned}
        \|V_1^{\varepsilon}\|_{W^{1, q}(B_{1/2})} &\leq C\int_{B_{2\overline{a}^{\varepsilon}}} \left|F^{\varepsilon}(y)-\beta^{\varepsilon}\right|\,dy\\
        &=C\int_{B_{2}} \left|(a^{\varepsilon})^{-2}\overline{F}^{\varepsilon}(z)-\beta^{\varepsilon}\right| (\overline{a}^{\varepsilon})^n\,dz\\
        &\leq C\int_{B_2} |\overline{F}_0(z)| \,dz+o(1) \leq C
    \end{aligned}
\end{equation*}
for any $1 \leq q <n/(n-1)$. Here the choice of $L^1$-data is optimal in the sense that if $p>1$, then $\|(F^{\varepsilon}-\beta^{\varepsilon})\chi_{B_{2\overline{a}^{\varepsilon}}}\|_{L^p(B_{1/2})}$ is not uniformly bounded with respect to $\varepsilon>0$. Therefore, $V_1^{\varepsilon}$ converges to some $V_1 \in W^{1, q}_0(B_{1/2})$ weakly in $W^{1, q}(B_{1/2})$ for $1<q \leq n/(n-1)$. 

On the other hand, for fixed $y \in B_{1/2} \setminus \{0\}$, it follows from \Cref{lem-green1} that for $\tilde{y} \in B_{2\overline{a}^{\varepsilon}}$,
\begin{equation*}
    0 \leq G(y, \tilde{y}) \leq C|y|^{2-n} \quad \text{and} \quad |G(y, 0)-G(y, \tilde{y})| \leq C\overline{a}^{\varepsilon}.
\end{equation*}
if $\varepsilon>0$ is sufficiently small. Thus, we obtain that 
\begin{equation*}
\begin{aligned}
V_1^\varepsilon(y) &= -\int_{B_{2\overline{a}^{\varepsilon}}} G(y, \tilde y) ( F^\varepsilon(\tilde y) -\beta^\varepsilon ) \, d\tilde y\\
&=-\int_{B_{2\overline{a}^{\varepsilon}}} G(y, 0) F^\varepsilon(\tilde y) \, d\tilde y+o(1)\\
&=-G(y, 0) \int_{B_2}  \overline{F}^\varepsilon(z) \, dz+o(1)\\
&=-G(y, 0) \left(\int_{B_2}  \overline{F}_0(z) \, dz + \int_{B_2} (\overline{F}^\varepsilon(z) - \overline{F}_0(z) )\, dz \right)+o(1)\\
&\rightarrow 0.
\end{aligned}
\end{equation*}
For the last convergence, we utilized the estimates \eqref{f0-vanish} and \eqref{eq-estimate-fe}. Hence, we conclude that $V_1^\varepsilon$ converges to $0$ and that $V_1 \equiv 0$.

(iii) Conclusion; uniqueness: We now recall from the definition of $V^{\varepsilon}$ that
\begin{equation*}
\begin{aligned}
    {V}^{\varepsilon} (y) &= \eta(y/\overline a^\varepsilon) ( W^\varepsilon (y) - \gamma_0 G(y)) + (1-\eta(y/\overline a^\varepsilon))(\varepsilon^{-2}\overline{W}_0(y/\overline a^\varepsilon)-\gamma_0\Gamma_0(y))\\
    &={V}_1^{\varepsilon} (y)+{V}_2^{\varepsilon} (y).
\end{aligned}
\end{equation*}
Let $(W^i, \beta_i)$, $i=1,2$, be two limits of $(W^\varepsilon, \beta^\varepsilon)$ up to  two different subsequences of $\varepsilon$. Without loss of generality, we may assume that $\beta_1>\beta_2$. By the convergence results obtained in the step (ii), we observe that for fixed $y \in B_{1/2} \setminus \{0\}$,
\begin{equation*}
    V_2^{1}(y)=W^1(y)-\gamma_0G(y) \quad \text{and} \quad  V_2^{2}(y)=W^2(y)-\gamma_0G(y),
\end{equation*}
where $V_2^i$ denotes the limits corresponding to $(W^{i}, \beta^i)$. It is easy to check that
\begin{enumerate}[(a)]
    \item $W^1$ and $W^2$ are $1$-periodic;
    
    \item $V_2^1, V_2^2 \in L^{\infty}(B_{1/2})$ and so  $\lim_{y \to 0}(\beta_1W^2(y)-\beta_2W^1(y))=\infty$;

    \item $a_{ij}(y)D_{ij}(\beta_1W^2(y)-\beta_2W^1(y))=0$ in $\mathbb{R}^n \setminus \mathbb{Z}^n$.
\end{enumerate}
Therefore, the minimum of $\beta_1W^2-\beta_2W^1$ must be attained at some point $y_0 \in Q_1 \setminus \{0\}$, which implies that $\beta_1W^2-\beta_2W^1$ is a constant by the strong minimum principle. It leads to a contradiction, and we conclude that $\beta_1=\beta_2$.
\end{proof}

\section{Proof of \texorpdfstring{\Cref{thm-main}}{Theorem 1.1} }\label{sec-homogenization}
The goal of this section is to complete the homogenization process by describing the convergence of $u_{\varepsilon}$ to $u$ in a suitable sense and by finding the homogenized equation satisfied by $u$. By scaling $W^{\varepsilon}$ found in \Cref{sec-correctors}, we have an $\varepsilon$-periodic function $w^{\varepsilon}$ satisfying
    \begin{equation*}
\left\{ \begin{aligned} 
a_{ij}(x/\varepsilon)D_{ij}w^{\varepsilon}(x)&=\beta^{\varepsilon} && \text{in $\mathbb{R}^n \setminus T_{a^{\varepsilon}}$} \\
w^{\varepsilon}&=1  && \text{on $\partial T_{a^{\varepsilon}}$}\\
\inf w^{\varepsilon}&=0.  &&
\end{aligned} \right.
\end{equation*}

 We now extend $w^{\varepsilon}$ to be $1$ in the holes $T_{a^{\varepsilon}}$ so that $w^{\varepsilon}$ becomes a continuous function defined in $\mathbb{R}^n$. Moreover, we show the following limiting property of $w^{\varepsilon}$, which is indeed related to the property $\inf W^{\varepsilon}=0$ in the fast variable.
\begin{lemma}\label{lem-convergence}
    The corrector $w^{\varepsilon}$	converges to $0$ ``away from holes" in the following sense:
    \begin{equation*}
        0 \leq \max_{\mathbb{R}^n \setminus T_{b^{\varepsilon}}} w^{\varepsilon}\leq C \varepsilon.
    \end{equation*}
\end{lemma}

\begin{proof}
An application of \Cref{lem-correctorbound} yields the estimate 
\begin{equation*}
    0\leq W^{\varepsilon}(y) \leq C|y-k|^{2-n},
\end{equation*}
where $y \in Q_1(k) \setminus B_{\overline{a}^{\varepsilon}}(k)$ for some $k \in \mathbb{Z}^n$. Then by scaling back, we have 
\begin{equation*}
    0 \leq w^{\varepsilon}(x) \leq C\varepsilon^n|x-\varepsilon k|^{2-n}.
\end{equation*}
Therefore, if $x \in \mathbb{R}^n \setminus T_{b^{\varepsilon}}$, then $|x-\varepsilon k| \geq \varepsilon^{\frac{n-1}{n-2}}$ and so the desired convergence follows.
\end{proof}

We now recall the uniform convergence up to subsequence obtained in \Cref{thm-convergence}. To be precise, we take a sequence $\{\varepsilon_m\}_{m \in \mathbb{N}}$ such that $\varepsilon_m \to 0$ and $\underline{u}_{\varepsilon_m} \to {u}$ uniformly in $\Omega$ as $m \to \infty$, where
\begin{equation*}
		\underline{u}_{\varepsilon}(x)=
		\begin{cases}
			u_{\varepsilon}(x) & \text{if $x \in Q_{\varepsilon}(\varepsilon k) \setminus B_{b^{\varepsilon}}(\varepsilon k)$ for some $k \in \mathbb{Z}^n$}\\
			\min_{\partial B_{b^{\varepsilon}}(\varepsilon k)} u_{\varepsilon} & \text{if $x \in B_{b^{\varepsilon}}(\varepsilon k)$ for some $k \in \mathbb{Z}^n$}.
	\end{cases}
	\end{equation*}
 Then by the definition of $\underline{u}_{\varepsilon}$, for any $x \in \Omega$, there exists a sequence $\{x_m\} \subset \Omega \setminus T_{b^{\varepsilon_m}}$ such that $x_m \to x$ and 
 \begin{equation*}
     {u}_{\varepsilon_m}(x_m)=\underline{u}_{\varepsilon_m}(x_m) \to {u}(x).
 \end{equation*}
 We note that the uniqueness of ${u}$ is unknown yet, i.e., ${u}$ may depend on the choice of $\varepsilon_m$. Indeed, our goal is to prove ${u}$ satisfies the homogenized equation stated in \Cref{thm-main} and so the uniqueness is guaranteed by the unique solvability of the homogenized problem.
	
	We are finally ready to prove \Cref{thm-main}. In the proof, we write $w_2^{\varepsilon}=w^{\varepsilon}$ to emphasize that it corresponds to the second part of the corrector.
	\begin{proof}[Proof of \Cref{thm-main}]
    (i) We first suppose by contradiction that ${u}$ is not a viscosity supersolution of \eqref{eq-homo}. In other words, there exists a quadratic polynomial $P$ touching $u$ from below at the origin (for simplicity) such that there exists $R>0$ satisfying $u(0)=P(0)$, $u\geq P$ in $B_{R}$, and 
		\begin{equation}\label{eq-homoeq}
			\overline{a}_{ij}D_{ij}P+\beta_0(\varphi(0)-P\\
			(0))_+ \geq (1+\beta_0)\eta >0 \quad \text{for some $\eta>0$}.
		\end{equation}
        Since $\varphi$ and $u$ are continuous in $\Omega$, we may assume that 
        \begin{equation}\label{eq-osc}
            \osc_{B_R} \, (\varphi-u) <\eta/2
        \end{equation}
        by choosing smaller $R>0$ if necessary.
  
		We then fix small $R>0$ and let
		\begin{equation*}
			\widetilde{P}(x):=P(x)-\frac{K}{2}\left(|x|^2-\frac{R^2}{2}\right), \quad\text{where $K=K(\eta) \in (0, 1)$ to be determined}
		\end{equation*}
		and 
		\begin{equation*}
			\widetilde{P}_m(x):=\widetilde{P}(x)+\varepsilon_m^2 w_{1, D^2\widetilde{P}}(x/\varepsilon_m)+((\varphi(0)-u(0))_+-\eta)w_2^{\varepsilon_m}(x).
		\end{equation*}
        
        Moreover, we choose a domain  
        \begin{equation*}
            D_{R, m}:=B_R \setminus \left(\bigcup_{k \in \mathcal{I}_{R, m}} B_{b^{\varepsilon_m}}(\varepsilon_m k) \right),
        \end{equation*}
        where 
        \begin{equation*}
            \mathcal{I}_{R, m}:=\{k \in \mathbb{Z}^n: \partial B_R \cap B_{b^{\varepsilon_m}}(\varepsilon_m k) \neq \varnothing\}.
        \end{equation*}
        We remark that the domain $D_{R, m}$ is perturbed from $B_R$ so that $\partial D_{R, m} \subset \Omega \setminus T_{b^{\varepsilon_m}}$ and $\partial D_{R, m} \subset B_{R} \setminus B_{3R/4}$ (if $m$ is sufficiently large).

        We claim that $\widetilde{P}_m$ is a viscosity subsolution in $D_{R, m} \setminus T_{a^{\varepsilon_m}}$ which is smaller than $u_{\varepsilon_m}$ on its boundary. For this purpose, by applying \Cref{lem:periodic}, \Cref{lem:periodic2} and recalling \eqref{eq-homoeq}, we observe that
		\begin{equation*}
			\begin{aligned}
				&a_{ij}(x/\varepsilon_m)D_{ij}\widetilde{P}_m(x)\\
				&=a_{ij}(x/\varepsilon_m)D_{ij}\widetilde{P}+a_{ij}(x/\varepsilon_m) (D_{ij}w_{1, D^2\widetilde{P}})(x/\varepsilon_m)+((\varphi(0)-u(0))_+-\eta)a_{ij}(x/\varepsilon_m)D_{ij}w_2^{\varepsilon_m}(x)\\
				&=\overline{a}_{ij}D_{ij}\widetilde{P}+\beta^{\varepsilon_m}(\varphi(0)-u(0))_+-\beta^{\varepsilon_m}\eta\\
				&\geq\overline{a}_{ij}D_{ij}\widetilde{P}-n\Lambda K+\beta^{\varepsilon_m}(\varphi(0)-u(0))_+-\beta^{\varepsilon_m}\eta\\
				& \geq \eta/2-n\Lambda K>0 \quad \text{in $\Omega \setminus T_{a^{\varepsilon_m}}$}.
			\end{aligned}
		\end{equation*}
		For the inequalities above, we choose sufficiently small $K>0$ and large $m$.
		
		On the other hand, for $x \in \partial T_{a^{\varepsilon_m}} \cap D_{R, m}$, we choose the unique $k \in \mathbb{Z}^n$ such that $x \in B_{b^{\varepsilon_m}}(\varepsilon_m k)$. Then we have
		\begin{equation*}
			\begin{aligned}
				\widetilde{P}_m(x) &\leq P(x)-\frac{K}{2}\left(|x|^2-\frac{R^2}{2}\right)+\varepsilon_m^2\|w_{1, D^2\widetilde{P}}\|_{{\infty}}+(\varphi(0)-u(0))_+-\eta \\
				&\leq u(x)+\frac{KR^2}{4}+\varepsilon_m^2\|w_{1, D^2\widetilde{P}}\|_{{\infty}}+(\varphi(0)-u(0))_+-\eta\\
				&\leq u(x)+(\varphi(0)-u(0))_+-\eta/2,
			\end{aligned}
		\end{equation*}		
        where we choose $m$ large enough and $K$ small enough. In the case of $\varphi(0) \geq u(0)$, the oscillation control \eqref{eq-osc} implies that
        \begin{equation*}
            \widetilde{P}_m(x) \leq \varphi(x) \leq u_{\varepsilon_m}(x).
        \end{equation*}
       In the case of $\varphi(0)<u(0)$, an application of the minimum principle to a supersolution $u_{\varepsilon_m}$ shows that
       \begin{equation*}
           u_{\varepsilon_m}(x) \geq \min_{\partial B_{b^{\varepsilon_m}}(\varepsilon_m k)} u_{\varepsilon_m}=\underline{u}_{\varepsilon_m}(x).
       \end{equation*}
       Since $\underline{u}_{\varepsilon_m}(x)$ converges to $u(x)$, we also have
       \begin{equation*}
           \widetilde{P}_m(x) \leq u(x)-\eta/2 \leq u_{\varepsilon_m}(x).
       \end{equation*}
       Moreover, for $x \in \partial D_{R, m}$, a similar argument gives that 
		\begin{equation*}
			\begin{aligned}
				\widetilde{P}_m(x) &\leq P(x)-\frac{KR^2}{32}+\varepsilon_m^2\|w_{1, D^2\widetilde{P}}\|_{{\infty}}+((\varphi(0)-u(0))_+-\eta)w_2^{\varepsilon_m}(x)\\
				&\leq u(x)-\frac{KR^2}{64}+((\varphi(0)-u(0))_+-\eta)w_2^{\varepsilon_m}(x).
			\end{aligned}
		\end{equation*}
        Since we choose $D_{R, m}$ so that $\partial D_{R, m} \subset \Omega \setminus T_{b^{\varepsilon_m}}$, we can apply \Cref{lem-convergence} to obtain that $w_2^{\varepsilon_m}(x)$ can be sufficiently small. Therefore, we conclude that
        \begin{equation*}
            \widetilde{P}_m(x) \leq \underline{u}_{\varepsilon_m}(x)={u}_{\varepsilon_m}(x),
        \end{equation*}
        which finishes the proof of the claim. We are now able to apply the comparison principle between $\widetilde{P}_m$ and $u_{\varepsilon_m}$ in $D_{R, m} \setminus T_{a^{\varepsilon_m}}$. We recall that there exists a sequence $\{x_m\} \subset \Omega \setminus T_{b^{\varepsilon_m}}$ such that $x_m \to 0$ and 
        \begin{equation*}
            u_{\varepsilon_m}(x_m)=\underline{u}_{\varepsilon_m}(x_m) \to u(0).
        \end{equation*}
        Thus, by taking the point $x_m$ in the inequality from the comparison principle, we have
		\begin{equation*}
			\widetilde{P}_m(x_m) \leq u_{\varepsilon_m}({x}_m)
		\end{equation*}
		or equivalently,
		\begin{equation*}
				P(x_m)+\frac{KR^2}{8}+\varepsilon_m^2 w_{1, D^2\widetilde{P}}(x_m/\varepsilon_m)+((\varphi(0)-u(0))_+-\eta)w_2^{\varepsilon_m}(x_m) \leq u_{\varepsilon_m}(x_m).
		\end{equation*}
        Letting $m \to \infty$ together with \Cref{lem-convergence}, we arrive at 
        \begin{equation*}
            P(0)+\frac{KR^2}{8} \leq u(0),
        \end{equation*}
        which leads to the contradiction. 

        (ii) We next suppose by contradiction that $u$ is not a viscosity subsolution of \eqref{eq-homo}. In other words, there exists a quadratic polynomial $P$ touching $u$ from above at the origin such that there exists $R>0$ satisfying $u(0)=P(0)$, $u \leq P$ in $B_{R}$, and 
		\begin{equation}\label{eq-homoeq2}
			\overline{a}_{ij}D_{ij}P+\beta_0(\varphi(0)-P\\
			(0))_+ \leq -(1+\beta_0)\eta <0 \quad \text{for some $\eta>0$}.
		\end{equation}
        As in the step (i), we fix $R>0$ satisfying \eqref{eq-osc} and choose the domain $D_{R, m}$. We then let
		\begin{equation*}
			\widetilde{P}(x):=P(x)+\frac{K}{2}\left(|x|^2-\frac{R^2}{2}\right), \quad\text{where $K=K(\eta) \in (0, 1)$ to be determined}
		\end{equation*}
		and 
		\begin{equation*}
			\widetilde{P}_m(x):=\widetilde{P}(x)+\varepsilon_m^2 w_{1, D^2\widetilde{P}}(x/\varepsilon_m)+((\varphi(0)-u(0))_++\eta)w_2^{\varepsilon_m}(x).
		\end{equation*}
        In a similar way as in the step (i), we observe that
		\begin{equation*}
			\begin{aligned}
				&a_{ij}(x/\varepsilon_m)D_{ij}\widetilde{P}_m(x)\\
				&=a_{ij}(x/\varepsilon_m)D_{ij}\widetilde{P}+a_{ij}(x/\varepsilon_m) (D_{ij}w_{1, D^2\widetilde{P}})(x/\varepsilon_m)+((\varphi(0)-u(0))_++\eta)a_{ij}(x/\varepsilon_m)D_{ij}w_2^{\varepsilon_m}(x)\\
				&=\overline{a}_{ij}D_{ij}\widetilde{P}+\beta^{\varepsilon_m}(\varphi(0)-u(0))_++\beta^{\varepsilon_m}\eta\\
				&\leq\overline{a}_{ij}D_{ij}P+n\Lambda K+\beta^{\varepsilon_m}(\varphi(0)-u(0))_++\beta^{\varepsilon_m}\eta\\
				& \leq -\eta/2+n\Lambda K<0 \quad \text{in $\Omega \setminus T_{a^{\varepsilon_m}}$}.
			\end{aligned}
		\end{equation*}

        On the other hand, for $x \in T_{a^{\varepsilon_m}} \cap D_{R, m}$, we have
		\begin{equation*}
			\begin{aligned}
				\widetilde{P}_m(x) &\geq P(x)+\frac{K}{2}\left(|x|^2-\frac{R^2}{2}\right)-\varepsilon_m^2\|w_{1, D^2\widetilde{P}}\|_{{\infty}}+(\varphi(0)-u(0))_++\eta \\
				&\geq u(x)-\frac{KR^2}{4}-\varepsilon_m^2\|w_{1, D^2\widetilde{P}}\|_{{\infty}}+(\varphi(0)-u(0))_++\eta\\
				&\geq u(x)+(\varphi(0)-u(0))+\eta/2,
			\end{aligned}
		\end{equation*}		
        where we choose $m$ large enough and $K$ small enough. Hence, the oscillation control \eqref{eq-osc} implies that $\widetilde{P}_m(x) \geq \varphi(x)$ for $x \in T_{a^{\varepsilon_m}} \cap D_{R, m}$.
        
        Moreover, for $x \in \partial D_{R, m}$, a similar argument gives that 
		\begin{equation*}
			\begin{aligned}
				\widetilde{P}_m(x) &\geq P(x)+\frac{KR^2}{32}-\varepsilon_m^2\|w_{1, D^2\widetilde{P}}\|_{{\infty}}\\
				&\geq u(x)+\frac{KR^2}{64}> \underline{u}_{\varepsilon_m}(x)={u}_{\varepsilon_m}(x).
			\end{aligned}
		\end{equation*}
        Let us now define the function
        \begin{equation*}
            v_{\varepsilon_m}:=
            \begin{cases}
                \min\{u_{\varepsilon_m}, \widetilde{P}_m\} & \text{in $D_{R, m}$}\\
                u_{\varepsilon_m} & \text{in $\Omega \setminus D_{R, m}$}.
            \end{cases}
        \end{equation*}
       In view of the previous observations, $v_{\varepsilon_m}$ is a well-defined viscosity supersolution of \eqref{problem-Le}. Since $u_{\varepsilon_m}$ is the least viscosity supersolution among such functions, we obtain that
       \begin{equation*}
           u_{\varepsilon_m} \leq v_{\varepsilon_m} \leq \widetilde{P}_m \quad \text{in $D_{R, m}$}.
       \end{equation*}
       By taking the point $x_m$ as in (i) and letting $m \to \infty$, we conclude that
       \begin{equation*}
           u(0) \leq P(0)-\frac{KR^2}{8},
       \end{equation*}
       which leads to the contradiction.
		
		(iii) In conclusion, $u$ is a viscosity solution of \eqref{eq-homo}. Since the Dirichlet problem admits at most one solution from the comparison principle, a limit function $u$ is indeed independent of the choice of a subsequence $\varepsilon_m$. The $L^p$-convergence of $u_{\varepsilon}$ immediately follows from the estimate
    \begin{equation*}
        \int_{\Omega}|u_{\varepsilon}-\underline{u}_{\varepsilon}|^p \,dx = \int_{T_{b^{\varepsilon}}}|u_{\varepsilon}-\underline{u}_{\varepsilon}|^p \,dx \leq C (b^{\varepsilon})^n \cdot \varepsilon^{-n}=C\varepsilon^{\frac{n}{n-2}}.
    \end{equation*}
	\end{proof}

We present a short remark on the critical value $\beta_0$.
\begin{remark}[Interpretation of $\beta_0$]
    In the case of $L=\Delta$, the critical value $\beta_0$ coincides with the capacity $\gamma_0=\mathrm{cap}_{\Delta}(B_1)$; see \cite{CL08, CM82a, CM82b} for instance. Thus, the natural question is to characterize the critical value $\beta_0$ in our variable coefficient setting. 
    
    We would like to point out that, in \Cref{sec-correctors}, the capacity $\gamma_0=\mathrm{cap}_{a_{ij}(0)}(B_1)$ played a crucial role in capturing the asymptotic behavior of the scaled corrector $\overline{W}^{\varepsilon}$. Nevertheless, it turns out that $\beta_0$ does not coincide with the quantity $\gamma_0$ in general.  To be precise, suppose that $a_{ij}(\cdot)$ is $1$-periodic and belongs to $C^2(\overline{Q_1})$. 
	Then by applying the integration by parts twice, together with the periodicity of $W^{\varepsilon}$ and $a_{ij}$, we find that
	\begin{equation}\label{eq-limitbeta}
		\begin{aligned}
		|Q_1 \setminus B_{\overline{a}^{\varepsilon}}| \cdot \beta^{\varepsilon}&=\int_{Q_1 \setminus B_{\overline{a}^{\varepsilon}}} a_{ij}(y) D_{ij}W^{\varepsilon}(y) \,dy\\
		&=-\int_{Q_1 \setminus B_{\overline{a}^{\varepsilon}}} D_ja_{ij} D_{i}W^{\varepsilon}+\underbrace{\int_{\partial Q_1}a_{ij}D_iW^{\varepsilon} \nu_j}_{=0}+\int_{\partial B_{\overline{a}^{\varepsilon}}}a_{ij}D_iW^{\varepsilon} \nu_j\\
		&=\underbrace{\int_{Q_1 \setminus B_{\overline{a}^{\varepsilon}}} D_{ij}a_{ij} W^{\varepsilon}}_{=:I_1}-\underbrace{\int_{\partial Q_1}D_ja_{ij}W^{\varepsilon} \nu_i}_{=0}-\underbrace{\int_{\partial B_{\overline{a}^{\varepsilon}}}D_ja_{ij}W^{\varepsilon} \nu_i}_{=:I_2}+\underbrace{\int_{\partial B_{\overline{a}^{\varepsilon}}}a_{ij}D_iW^{\varepsilon} \nu_j}_{=:I_3}.
		\end{aligned}
	\end{equation}
Since $W^{\varepsilon}=\varepsilon^{-2}$ on $\partial B_{\overline{a}^{\varepsilon}}$, we have
\begin{equation*}
	\begin{aligned}
		I_2=\varepsilon^{-2}\int_{\partial B_{\overline{a}^{\varepsilon}}}D_ja_{ij} \nu_i=\varepsilon^{-2}\left[\int_{Q_1 \setminus B_{\overline{a}^{\varepsilon}}} D_{ij}a_{ij} - \int_{\partial Q_1}D_ja_{ij}\nu_i\right]=\varepsilon^{-2}\int_{Q_1 \setminus B_{\overline{a}^{\varepsilon}}} D_{ij}a_{ij}.
	\end{aligned}
\end{equation*}
The periodicity of $a_{ij} \in C^2$ again yields that 
\begin{equation*}
	|I_2|=\varepsilon^{-2}\left|\int_{Q_1 \setminus B_{\overline{a}^{\varepsilon}}} D_{ij}a_{ij}\right|=\varepsilon^{-2}\left|\int_{B_{\overline{a}^{\varepsilon}}} D_{ij}a_{ij}\right| \lesssim \varepsilon^{\frac{4}{n-2}} \to 0 \quad \text{when $\varepsilon \to 0$}.
\end{equation*}
We next claim that
\begin{equation*}
	\lim_{\varepsilon \to 0}I_3=\gamma_0.
\end{equation*}
If we recall a rescaled function $\overline{W}^{\varepsilon}(z)=\varepsilon^2 W^{\varepsilon}(\overline{a}^{\varepsilon}z)$, then we may write
	\begin{equation*}
		\int_{\partial B_{\overline{a}^{\varepsilon}}}a_{ij}(y)D_iW^{\varepsilon}(y) \nu_j \,d\sigma_y=\int_{\partial B_{1}}a_{ij}(\overline{a}^{\varepsilon}z)D_i\overline{W}^{\varepsilon}(z) \nu_j \,d\sigma_z.
	\end{equation*}
It also follows from the proof of \Cref{thm-uniquebeta} that for any $\alpha \in (0,1)$,
\begin{equation*}
	\|\overline{W}^{\varepsilon}-\overline{W}_0\|_{C^{1, \alpha}(\overline{B_2 \setminus B_{1}})}=o(1)
\end{equation*}	
and so $\|\overline{W}^{\varepsilon}\|_{C^{1, \alpha}(\overline{B_2 \setminus B_{1}})}$ is uniformly bounded with respect to $\varepsilon>0$. Thus, we conclude that
\begin{equation*}
	\begin{aligned}
			\left|\int_{\partial B_{\overline{a}^{\varepsilon}}}a_{ij}(y)D_iW^{\varepsilon}(y) \nu_j \,d\sigma_y-\gamma_0 \right| &\leq \int_{\partial B_{1}}|a_{ij}(\overline{a}^{\varepsilon}z)-a_{ij}(0)| |D_i\overline{W}^{\varepsilon}(z)| \,d\sigma_z\\
			&\quad+\int_{\partial B_{1}}|a_{ij}(0)| |D_i\overline{W}^{\varepsilon}(z)-D_i\overline{W}_0(z)|  \,d\sigma_z\\
			&\to 0 \quad \text{when $\varepsilon \to 0$}. 
	\end{aligned}
\end{equation*}
It only remains to investigate the term $I_1$. The limit $\lim_{\varepsilon \to 0}I_1$ is well-defined, since all other terms in \eqref{eq-limitbeta} have limits when $\varepsilon \to 0$.

We would like to provide a special coefficient matrix $a_{ij}(\cdot)$ for which $\lim_{\varepsilon \to 0}I_1 \neq 0$. For $\delta \in (0, 1/8)$ to be determined soon, we define a cut-off function $g_{\delta} \in C_c^{\infty}(B_{\delta})$ such that $g_{\delta} =1$ in $B_{\delta/2}$ and $0 \leq g_{\delta} \leq 1$ in $B_{\delta}$. We let the $1$-periodic function $\psi \in C^{\infty}(\mathbb{R}^n)$ solve
\begin{equation*}
	\Delta \psi(y)=g_{\delta}(y)-g_{\delta}(y-e_1/4) \quad \text{in $Q_1$}
\end{equation*}
with the condition $\inf_{\mathbb{R}^n} \psi=0$. We note that $\|\psi\|_{L^{\infty}}$ is uniformly bounded with respect to $\delta>0$. Then we define a coefficient matrix $a_{ij}(\cdot)$ by
\begin{equation*}
	a_{ij}(y):=(2+\psi(y))\delta_{ij}.
\end{equation*}
It is immediate to check that $a_{ij}(\cdot)$ is uniformly elliptic, $1$-periodic, smooth, and
\begin{equation*}
	D_{ij}a_{ij}=\Delta \psi.
\end{equation*}
Moreover, in view of \Cref{lem-correctorbound}, we observe that:

(i) there exist universal constants $c>0$ and $r_1 >0$ such that
\begin{equation*}
	W^{\varepsilon}(y) \geq c|y|^{2-n} \quad \text{for $\overline{a}^{\varepsilon} \leq |y| \leq r_1$};
\end{equation*}

(ii) $0 \leq W^{\varepsilon}(y) \leq C$ for $y \in Q_1 \setminus B_{3/8}$, where $C>0$ is independent of $\varepsilon>0$.

Therefore, it turns out that
\begin{equation*}
	\begin{aligned}
		\int_{Q_1 \setminus B_{\overline{a}^{\varepsilon}}} D_{ij}a_{ij} W^{\varepsilon}= \int_{Q_1 \setminus B_{\overline{a}^{\varepsilon}}} \Delta \psi \cdot W^{\varepsilon}&\geq \int_{B_{\delta/2} \setminus B_{\overline{a}^{\varepsilon}}} W^{\varepsilon}-\int_{B_{\delta}(e_1/4)}  W^{\varepsilon} \\
		&\geq C_1\delta^2-C_2\delta^n>0,
	\end{aligned}
\end{equation*}
if $\delta$ is chosen to be sufficiently small.
To this end, the limit of $\beta^{\varepsilon}$ need not coincide with the value $\beta_0$ in general. Note that we have $\beta_0=\gamma_0$ provided that $a_{ij}(\cdot) \in C^2(\overline{Q_1})$ and $\lim_{\varepsilon \to 0}I_1=0$; in particular, such conditions hold when $a_{ij}(\cdot)\equiv a_{ij}(0)$.

In fact, we conjecture that the critical value $\beta_0$ in our framework coincides with $\mathrm{cap}_{\overline{a}_{ij}}(B_1)$ that is the capacity for the homogenized operator.

\end{remark}

\section{Homogenization for non-critical hole sizes}\label{sec-noncritical}
In this last section, we provide homogenization results for non-critical hole sizes by utilizing the corrector $W^{\varepsilon}$ constructed in \Cref{sec-correctors}. To be precise, we now consider holes $T_{a^{\varepsilon^\alpha}}$ with non-critical size
\begin{equation*}
	{a^{\varepsilon^\alpha}}=\varepsilon^{\frac{n}{n-2}\alpha} \quad \text{for some $0<\alpha \neq 1$}.
\end{equation*}
We note that if we let
\begin{equation*}
    \overline{\alpha}:=\frac{n}{2}\alpha-\frac{n-2}{2},
\end{equation*}
then we have the relation
\begin{equation*}
    \overline a^{\varepsilon^{\overline \alpha}}=a^{\varepsilon^{\overline \alpha}}/\varepsilon^{\overline \alpha}=a^{\varepsilon^\alpha} / \varepsilon.
\end{equation*}
By \Cref{cor-critical}, $W^{\varepsilon^{\overline \alpha}}$ satisfies
\begin{equation}\label{eq-for-W-ep-alpha-large-hole}
	\left\{ \begin{aligned} 
		a_{ij}(y)D_{ij}W^{\varepsilon^{\overline \alpha}}(y)&=\beta^{\varepsilon^{\overline \alpha}} && \text{in $\mathbb{R}^n \setminus \overline{T}_{\overline{a}^{\varepsilon^{\overline \alpha}}}$} \\
		W^{\varepsilon^{\overline \alpha}}&=\varepsilon^{-2\overline \alpha}  && \text{on $\partial \overline{T}_{\overline{a}^{\varepsilon^{\overline \alpha}}}$}\\
		\inf W^{\varepsilon^{\overline \alpha}}
&=0  && 
	\end{aligned} \right.
\end{equation}
We first consider the case of $\alpha>1$, that is, when the size of holes is smaller than the critical one. We define correctors by 
\begin{equation}\label{def-of-corrector-from-W-ep-in-small-hole}
\widehat W^\varepsilon(y) := \varepsilon^{2(\overline \alpha-1)}W^{\varepsilon^{\overline \alpha}}(y) \quad \text{and} \quad
\widehat w^\varepsilon(x) := \varepsilon^2 \widehat W^\varepsilon(x/\varepsilon).
\end{equation}
Then by \eqref{eq-for-W-ep-alpha-large-hole}, \eqref{def-of-corrector-from-W-ep-in-small-hole}, and \Cref{lem-convergence}, we have the following results. 
\begin{lemma}\label{lem-sub-critical-correctibility-condition}
Let $\alpha>1$. Then the following hold:
\begin{enumerate}[(i)]
\item $a_{ij}(x/\varepsilon) D_{ij} \widehat w^\varepsilon(x) = \varepsilon^{2(\overline{\alpha}-1)} \beta^{\varepsilon^{\overline{\alpha}}} =:\widehat \beta^\varepsilon \rightarrow 0$ as $\varepsilon \to 0$.

\item $\widehat w^\varepsilon = 1$ on $\partial T_{a^{\varepsilon^{\alpha}}}$.

\item $\inf \widehat w^\varepsilon =0$.

\item $\max_{\mathbb{R}^n \setminus T_{\sqrt{a^{\varepsilon^{\alpha} }\varepsilon} }} \widehat w^{\varepsilon}=o(1)$.
\end{enumerate}
\end{lemma}

By \Cref{thm-convergence}, \Cref{lem-sub-critical-correctibility-condition}, and an argument similar to the proof of \Cref{thm-main} with the corrector $\widehat{w}^{\varepsilon}$, the limit function $u$ satisfies
\begin{equation*}
    \overline{a}_{ij} D_{ij}u+\widehat{\beta}_0 (\varphi-u)_+ =0
\end{equation*}
where
\begin{equation*}
    \widehat{\beta}_0=\lim_{\varepsilon\to 0}\widehat{\beta}^{\varepsilon}=0. 
\end{equation*}
Hence, we deduce the following homogenization result when $\alpha>1$.
\begin{theorem}\label{sub-critical-case-small-ball}
	Let $\alpha>1$ and let
 \begin{equation*}
\varphi_{\varepsilon^{\alpha}}(x):=
\left\{ \begin{array}{ll} 
\varphi(x) & \text{if $x \in T_{a^{\varepsilon^{\alpha}}}$} \\
0 & \text{otherwise}.
\end{array} \right.
\end{equation*}
Suppose that $u_{\varepsilon}$ is the least viscosity supersolution of \eqref{problem-Le} with the obstacle $\varphi_{\varepsilon}$ being replaced by $\varphi_{\varepsilon^{\alpha}}$. Then there exists a function $u \in C(\overline{\Omega})$ such that $\underline{u}_{\varepsilon} \to u$ uniformly in $\Omega$. Moreover,  there exists a uniformly elliptic constant matrix $\overline{a}_{ij}$ such that $u$ is a viscosity solution of
		\begin{equation*}
		\left\{ \begin{aligned}
		\overline{a}_{ij} D_{ij}u&=0 && \text{in $\Omega$} \\
		u&=0 && \text{on $\partial \Omega$}.
		\end{aligned} \right.
		\end{equation*}
\end{theorem}

We next consider the case of $\alpha \in ((n-2)/n,1)$, that is, when the size of holes is larger than the critical one. We define correctors by
\begin{equation}\label{def-of-corrector-from-W-ep-in-large-hole}
\widetilde W^\varepsilon(y):= W^{\varepsilon^{\overline \alpha}}(y) + \frac{1}{\varepsilon^2} - \frac{1}{\varepsilon^{2\overline\alpha}} \quad \mbox{and} \quad \widetilde w^\varepsilon(x) := \varepsilon^2 \widetilde W^\varepsilon(x/\varepsilon).
\end{equation}
Then by \eqref{eq-for-W-ep-alpha-large-hole}, \eqref{def-of-corrector-from-W-ep-in-large-hole}, and \Cref{lem-convergence}, we have the following results.
\begin{lemma}\label{eq-corrector-subcritical-case}
Let $(n-2)/n<\alpha<1$. Then the following hold:
\begin{enumerate}[(i)]
\item $a_{ij}(x/\varepsilon) D_{ij} \widetilde w^\varepsilon(x) = \beta^{\varepsilon^{\overline{\alpha}}} \rightarrow \beta_0$ as $\varepsilon\to 0$.

\item $\widetilde w^\varepsilon = 1$ on $\partial T_{a^{\varepsilon^{\alpha}}}$.

\item $1 - \varepsilon^{2-2\overline\alpha} \le  \widetilde w^\varepsilon(x) \le 1 $ in $\mathbb{R}^n \setminus T_{a^{\varepsilon^{\alpha}}}$.
\end{enumerate}
\end{lemma}

\begin{theorem}
Let $(n-2)/n<\alpha<1$ and let
 \begin{equation*}
\varphi_{\varepsilon^{\alpha}}(x):=
\left\{ \begin{array}{ll} 
\varphi(x) & \text{if $x \in T_{a^{\varepsilon^{\alpha}}}$} \\
0 & \text{otherwise}.
\end{array} \right.
\end{equation*}
Suppose that $u_{\varepsilon}$ is the least viscosity supersolution of \eqref{problem-Le} with the obstacle $\varphi_{\varepsilon}$ being replaced by $\varphi_{\varepsilon^{\alpha}}$. Then there exists a function $u \in C(\overline{\Omega})$ such that ${u}_{\varepsilon} \to u$ uniformly in $\Omega$. Moreover, the limit $u$ is a least viscosity super solution of 
\begin{equation*}
		\left\{ \begin{aligned}
		\overline{a}_{ij} D_{ij}u&\leq 0 && \text{in $\Omega$} \\
		u&=0 && \text{on $\partial \Omega$}\\
  u&\geq\varphi && \text{in $\Omega$}.
		\end{aligned} \right.
		\end{equation*}
\end{theorem}
\begin{proof}
By following similar arguments as in \Cref{sub-critical-case-small-ball}, it is enough to show that $u \geq \varphi$ in $\Omega$. Indeed, we choose a constant $K$ so that 
\begin{equation*}
K>\frac{2\left\|a_{ij}D_{ij}\varphi\right\|_{L^{\infty}(\Omega)}}{\beta_0}+1,
\end{equation*}
and define a function 
\begin{equation*}
\psi^{\varepsilon}(x):=\varphi(x)+K\left(\widetilde{w}^{\varepsilon}(x)-1\right)
\end{equation*}
where $\widetilde{w}^{\varepsilon}$ is given by \Cref{eq-corrector-subcritical-case}. Then for sufficiently small $\varepsilon>0$ we have
\begin{equation*}
a_{ij}D_{ij}\psi^{\varepsilon}\geq a_{ij}D_{ij}\varphi+\frac{K\beta_0}{2}\geq 0 \qquad \mbox{in $\Omega\setminus T_{a^{\varepsilon^{\alpha}}}$}
\end{equation*}
and
\begin{equation*}
    \psi^{\varepsilon}(x) \leq u_{\varepsilon}(x) \qquad \mbox{on $\partial (\Omega\setminus T_{a^{\varepsilon^{\alpha}}})$}.
\end{equation*}
Thus, by comparison principle together with \Cref{eq-corrector-subcritical-case}, we have
\begin{equation*}
    u_{\varepsilon}(x)\geq \psi^{\varepsilon}(x)=\varphi(x)+K\left(\widetilde{w}^{\varepsilon}(x)-1\right)\geq \varphi(x)-K\varepsilon^{2-2\overline\alpha} \quad \mbox{in $\Omega\setminus T_{a^{\varepsilon^{\alpha}}}$},
\end{equation*}
and so the desired obstacle condition for $u$ follows.
\end{proof}

\bibliographystyle{abbrv}

\begin{thebibliography}{10}
	
	\bibitem{AKM17}
	S.~Armstrong, T.~Kuusi, and J.-C. Mourrat.
	\newblock The additive structure of elliptic homogenization.
	\newblock {\em Invent. Math.}, 208(3):999--1154, 2017.
	
	\bibitem{AL17}
	S.~Armstrong and J.~Lin.
	\newblock Optimal quantitative estimates in stochastic homogenization for
	elliptic equations in nondivergence form.
	\newblock {\em Arch. Ration. Mech. Anal.}, 225(2):937--991, 2017.
	
	\bibitem{AL89}
	M.~Avellaneda and F.-H. Lin.
	\newblock Compactness methods in the theory of homogenization. {II}.
	{E}quations in nondivergence form.
	\newblock {\em Comm. Pure Appl. Math.}, 42(2):139--172, 1989.
	
	\bibitem{Bau84}
	P.~Bauman.
	\newblock Positive solutions of elliptic equations in nondivergence form and
	their adjoints.
	\newblock {\em Ark. Mat.}, 22(2):153--173, 1984.
	
	\bibitem{BLP78}
	A.~Bensoussan, J.-L. Lions, and G.~Papanicolaou.
	\newblock {\em Asymptotic analysis for periodic structures}, volume~5 of {\em
		Studies in Mathematics and its Applications}.
	\newblock North-Holland Publishing Co., Amsterdam-New York, 1978.
	
	\bibitem{CL08}
	L.~Caffarelli and K.-A. Lee.
	\newblock Viscosity method for homogenization of highly oscillating obstacles.
	\newblock {\em Indiana Univ. Math. J.}, 57(4):1715--1741, 2008.
	
	\bibitem{CC95}
	L.~A. Caffarelli and X.~Cabr\'e.
	\newblock {\em Fully nonlinear elliptic equations}, volume~43 of {\em American
		Mathematical Society Colloquium Publications}.
	\newblock American Mathematical Society, Providence, RI, 1995.
	
	\bibitem{CM09}
	L.~A. Caffarelli and A.~Mellet.
	\newblock Random homogenization of an obstacle problem.
	\newblock {\em Ann. Inst. H. Poincar\'e{} C Anal. Non Lin\'eaire},
	26(2):375--395, 2009.
	
	\bibitem{CS10}
	L.~A. Caffarelli and P.~E. Souganidis.
	\newblock Rates of convergence for the homogenization of fully nonlinear
	uniformly elliptic pde in random media.
	\newblock {\em Invent. Math.}, 180(2):301--360, 2010.
	
	\bibitem{CSW05}
	L.~A. Caffarelli, P.~E. Souganidis, and L.~Wang.
	\newblock Homogenization of fully nonlinear, uniformly elliptic and parabolic
	partial differential equations in stationary ergodic media.
	\newblock {\em Comm. Pure Appl. Math.}, 58(3):319--361, 2005.
	
	\bibitem{CM82a}
	D.~Cioranescu and F.~Murat.
	\newblock Un terme \'etrange venu d'ailleurs.
	\newblock In {\em Nonlinear partial differential equations and their
		applications. {C}oll\`ege de {F}rance {S}eminar, {V}ol. {II} ({P}aris,
		1979/1980)}, volume~60 of {\em Res. Notes in Math.}, pages 98--138, 389--390.
	Pitman, Boston, Mass.-London, 1982.
	
	\bibitem{CM82b}
	D.~Cioranescu and F.~Murat.
	\newblock Un terme \'etrange venu d'ailleurs. {II}.
	\newblock In {\em Nonlinear partial differential equations and their
		applications. {C}oll\`ege de {F}rance {S}eminar, {V}ol. {III} ({P}aris,
		1980/1981)}, volume~70 of {\em Res. Notes in Math.}, pages 154--178,
	425--426. Pitman, Boston, Mass.-London, 1982.
	
	\bibitem{CIL92}
	M.~G. Crandall, H.~Ishii, and P.-L. Lions.
	\newblock User's guide to viscosity solutions of second order partial
	differential equations.
	\newblock {\em Bull. Amer. Math. Soc. (N.S.)}, 27(1):1--67, 1992.
	
	\bibitem{DEK18}
	H.~Dong, L.~Escauriaza, and S.~Kim.
	\newblock On {$C^1$}, {$C^2$}, and weak type-{$(1,1)$} estimates for linear
	elliptic operators: part {II}.
	\newblock {\em Math. Ann.}, 370(1-2):447--489, 2018.
	
	\bibitem{DK17}
	H.~Dong and S.~Kim.
	\newblock On {$C^1$}, {$C^2$}, and weak type-{$(1,1)$} estimates for linear
	elliptic operators.
	\newblock {\em Comm. Partial Differential Equations}, 42(3):417--435, 2017.
	
	\bibitem{DKL23}
	H.~Dong, S.~Kim, and S.~Lee.
	\newblock Note on {G}reen's functions of non-divergence elliptic operators with
	continuous coefficients.
	\newblock {\em Proc. Amer. Math. Soc.}, 151(5):2045--2055, 2023.
	
	\bibitem{Eva89}
	L.~C. Evans.
	\newblock The perturbed test function method for viscosity solutions of
	nonlinear {PDE}.
	\newblock {\em Proc. Roy. Soc. Edinburgh Sect. A}, 111(3-4):359--375, 1989.
	
	\bibitem{Eva92}
	L.~C. Evans.
	\newblock Periodic homogenisation of certain fully nonlinear partial
	differential equations.
	\newblock {\em Proc. Roy. Soc. Edinburgh Sect. A}, 120(3-4):245--265, 1992.
	
	\bibitem{GT83}
	D.~Gilbarg and N.~S. Trudinger.
	\newblock {\em Elliptic partial differential equations of second order}, volume
	224 of {\em Grundlehren der mathematischen Wissenschaften [Fundamental
		Principles of Mathematical Sciences]}.
	\newblock Springer-Verlag, Berlin, second edition, 1983.
	
	\bibitem{GW82}
	M.~Gr\"{u}ter and K.-O. Widman.
	\newblock The {G}reen function for uniformly elliptic equations.
	\newblock {\em Manuscripta Math.}, 37(3):303--342, 1982.
	
	\bibitem{HK20}
	S.~Hwang and S.~Kim.
	\newblock Green's function for second order elliptic equations in
	non-divergence form.
	\newblock {\em Potential Anal.}, 52(1):27--39, 2020.
	
	\bibitem{JKO94}
	V.~V. Jikov, S.~M. Kozlov, and O.~A. Ole{\u i}nik.
	\newblock {\em Homogenization of differential operators and integral
		functionals}.
	\newblock Springer-Verlag, Berlin, 1994.
	\newblock Translated from the Russian by G. A. Yosifian [G. A. Iosif\cprime
	yan].
	
	\bibitem{KL11}
	S.~Kim and K.-A. Lee.
	\newblock Viscosity method for homogenization of parabolic nonlinear equations
	in perforated domains.
	\newblock {\em J. Differential Equations}, 251(8):2296--2326, 2011.
	
	\bibitem{KL16}
	S.~Kim and K.-A. Lee.
	\newblock Higher order convergence rates in theory of homogenization: equations
	of non-divergence form.
	\newblock {\em Arch. Ration. Mech. Anal.}, 219(3):1273--1304, 2016.
	
	\bibitem{KL22}
	S.~Kim and K.-A. Lee.
	\newblock Uniform estimates in periodic homogenization of fully nonlinear
	elliptic equations.
	\newblock {\em Arch. Ration. Mech. Anal.}, 243(2):697--745, 2022.
	
	\bibitem{LL23}
	K.-A. Lee and S.-C. Lee.
	\newblock Viscosity method for random homogenization of fully nonlinear
	elliptic equations with highly oscillating obstacles.
	\newblock {\em Adv. Nonlinear Anal.}, 12(1):266--303, 2023.
	
	\bibitem{LY12}
	K.-A. Lee and M.~Yoo.
	\newblock The viscosity method for the homogenization of soft inclusions.
	\newblock {\em Arch. Ration. Mech. Anal.}, 206(1):297--332, 2012.
	
	\bibitem{Tan12}
	L.~Tang.
	\newblock Random homogenization of {$p$}-{L}aplacian with obstacles in
	perforated domain.
	\newblock {\em Comm. Partial Differential Equations}, 37(3):538--559, 2012.
	
\end{thebibliography}

\end{document}